\documentclass[dvipsnames]{article}

\usepackage[latin1, utf8]{inputenc}
\usepackage[T1]{fontenc}
\usepackage[english]{babel}
\usepackage{mathrsfs} 
\usepackage[numbers,sort]{natbib}
\usepackage{comment}
\usepackage{todonotes}
\usepackage{authblk}
\usepackage{hyperref}
\usepackage{amsmath}
\usepackage{amssymb}
\usepackage{amsthm}
\usepackage{enumitem}
\usepackage{stmaryrd}
\usepackage[toc,page]{appendix}
\usepackage[left=2.5cm,top=3cm,right=2.5cm,bottom=3cm,bindingoffset=0.5cm]{geometry}
\usepackage{xcolor}
\newcommand{\R}{\mathbb{R}}

\newcommand{\U}{\mathcal{U}}
\newcommand{\tar}{y_f}
\newcommand{\adjtar}{p_f}
\newcommand{\mL}{L|\Omega|}
\newcommand{\ie}{\textit{i.e.}, }

\newcommand{\fae}{\mathrm{for \ a.e.} \ }
\DeclareMathOperator*{\argmax}{arg\,max}
\DeclareMathOperator*{\argmin}{arg\,min}

\newcommand {\e}  {\varepsilon}
\newcommand {\dualfct}{J_{T,\e}}
\newcommand {\cost}{F_{T}}

\newcommand{\vertiii}[1]{{\vert\kern-0.25ex\vert\kern-0.25ex\vert #1 
    \vert\kern-0.25ex\vert\kern-0.25ex\vert}}
    
\newtheorem{prpstn}{Proposition}[section]
\newtheorem{lmm}[prpstn]{Lemma}
\newtheorem{thrm}[prpstn]{Theorem}
\newtheorem{informalth}{Theorem}

\newtheorem{dfntn}[prpstn]{Definition}
\newtheorem{crllr}[prpstn]{Corollary}
\newtheorem{rmrk}[prpstn]{Remark}

\everymath{\displaystyle}
\title{Approximate control of parabolic equations with on-off shape controls by Fenchel duality}
\author[a]{Camille Pouchol}
\author[b,d]{Emmanuel Trélat}
\author[c]{Christophe Zhang}
\affil[a]{Laboratoire MAP5 UMR 8145,
Université Paris Cité, 75006 Paris, France. Email address: camille.pouchol@u-paris.fr}
\affil[b]{Sorbonne Universit\'{e}, Universit\'{e} de Paris, CNRS, Laboratoire Jacques-Louis Lions, 75005 Paris, France. Email address: emmanuel.trelat@sorbonne-universite.fr}
\affil[c]{Université de Lorraine, CNRS, Inria, IECL, F-54000 Nancy, France. Email address: christophe.zhang@polytechnique.org}
\affil[d]{CAGE, INRIA, Paris, France.}
\date{\empty}

\begin{document}
\maketitle

\begin{abstract}
We consider the internal control of linear parabolic equations through \textit{on-off shape} controls, \ie controls of the form $M(t) \chi_{\omega(t)}$ with $M(t) \geq 0$ and $\omega(t)$ with a prescribed maximal measure. 

We establish small-time approximate controllability towards all possible final states allowed by the comparison principle with nonnegative controls. We manage to build controls with constant amplitude~$M(t) \equiv \overline{M}$. 
In contrast, if the moving control set $\omega(t)$ is confined to evolve in some region of the whole domain, we prove that approximate controllability fails to hold for small times.

The method of proof is constructive. Using Fenchel-Rockafellar duality and the bathtub principle, the on-off shape control is obtained as the bang-bang solution of an optimal control problem, which we design by relaxing the constraints.

%We show that optimal controls are bang-bang, which here amounts to shape controls.
%Finally, we formulate an equivalent minimal time problem, and prove that its solutions are shapes as well. 

Our optimal control approach is outlined in a rather general form for linear constrained control problems, paving the way for generalisations and applications to other PDEs and constraints.

\end{abstract}
\tableofcontents

\begin{comment}

Questions restantes
\begin{itemize}
  \item Ref pour contrôlabilité approchée de la chaleur par Holmgren? (historique)
  \item Terminologie on-off à préciser et fixer
  \item Semigroupes régularisants ? Terminologie ? \textit{Immediately differentiable}, propriété plus générique : \textit{analytic semigroups}.
\item Propriétés requises pour $\Omega$ :
        \begin{itemize}
            \item Domaine des opérateurs elliptiques $H^2\cap H^1_0$ : bord $C^2$.
            \item Analyticité : voir avec Emmanuel, a priori rien
            \item Pighin ($C^2$ utilisé pour régularité jusqu'au bord)
            \item On dirait qu'il faut $C^2$ pour dire que le semigroupe généré par le Laplacien Dirichlet est analytique.
        \end{itemize}
 
\end{itemize}
\end{comment}

\section{Introduction}
\subsection{Constrained internal control} 
This article is devoted to
%We consider
the internal approximate controllability problem at time $T>0$ for linear parabolic equations on a domain $\Omega$ by means of \textit{on-off shape controls}, \ie internal controls taking the form 
\[\forall t \in (0,T),\, \forall x \in \Omega, \quad u(t, x) = M(t) \chi_{\omega(t)}(x),\]
where,  at a given time $t \in (0,T)$,  \begin{itemize}
    \item $M(t) > 0$ is the \textbf{nonnegative} amplitude of the control
    \item $\chi_{\omega(t)}$ is the characteristic function of the set $\omega(t)\subset \Omega$, \ie $\chi_{\omega(t)}(x):= \begin{cases} 1 & \text{ if } x \in \omega(t), \\
    0 & \text{ otherwise}\end{cases}.$
\end{itemize}
Both the amplitude and location may be subject to constraints. 
This problem is a paradigmatic simplification of many practical situations where one can act on a complex system with on-off devices that can be moved in time, while their shape can also be modified.

Along the introduction, we expose our results for general operators $A$, while first illustrating them in the case of the controlled linear heat equation with Dirichlet boundary conditions
\begin{equation}\label{nonnegative-controllability-heat}
\left\{\begin{aligned}
    y_t - \Delta y &= u \ \textrm{in} \ \Omega, \\
    y& =0 \ \textrm{on} \ \partial \Omega, \\
    y(0) &=  y_0  \ \textrm{in} \ \Omega.
    \end{aligned}\right.
\end{equation}
In this setting, $\Omega$ is an open connected bounded subset of $\R^d$, with $C^2$ boundary, and $y_0 \in L^2(\Omega)$.

%with state space $L^2(\Omega)$, 
%by means of an internal \textbf{nonnegative} control $u \in L^2(0,T;L^2(\Omega))$. Here, $T>0$, $\Omega \subset \R^d$ is a smooth domain, $y_0 \in L^2(\Omega)$ and $(A, D(A))$ is a given operator generating a $C_0$ semigroup over $L^2(\Omega)$.

%The goal is to investigate approximate controllability for~\eqref{nonnegative-controllability} in arbitrary time $T>0$, in two contexts:
%\begin{itemize}
%    \item first, in the setting of nonnegative controls acting on the whole domain~$\Omega$, without any further assumption on the operator,
%    \item second, for controls satisfying additional constraints, at the expense of restricting ourselves to an appropriate subclass of operators.
%\end{itemize}
%The aforementioned subclass covers a family of elliptic operators with Dirichlet boundary conditions, but other operators as well.% However, our results also cover ... fractional heat + hyperbolic?
%, which leads to the problem of nonnegative controllability for the heat equation with Dirichlet boundary conditions. 

\paragraph{Control without constraints.}

When constraints are removed, generic parabolic equations are well-known to be approximately controllable~\cite{aronszajn1956unique, russell_1978}, and even null-controllable~\cite{Emanuilov_1995, lebeau1995controle} in arbitrarily small time by means of internal controls, acting only on an arbitrary fixed measurable subset $\omega \subset \Omega$ of positive measure.

This more precisely means that for any time $T>0$, any measurable set $\omega \subset \Omega$ of positive measure, any $\e>0$, any $y_0 \in L^2(\Omega)$ and target $\tar \in L^2(\Omega)$, there holds
\[\exists u \in L^2((0,T) \times \Omega), \, \text{ such that }\; \forall t \in (0,T),\; \mathrm{supp}(u(t, \cdot)) \subset \omega \text{ and }  \|y(T) - \tar \|_{L^2(\Omega)} \leq \e,\]
where $\mathrm{supp}(u)$ refers to the essential support of a function $u \in L^2(\Omega)$.

\paragraph{Constrained control.}
In view of applications where unilateral or bilateral or $L^\infty$ constraints naturally appear, constrained controllability has been an active area of research \cite{brammer1972controllability,ahmed_finite-time_1985, son_unified_1990}, whether in finite or infinite dimension.

In various contexts, control constraints have been shown to lead to controllability obstructions, even for unilateral constraints. Some states are out of reach, regardless of how large $T>0$ may be~\cite{Pouchol2019, Ruiz2020}. 
On the other hand, some states are reachable but only for $T$ large enough: constraints may lead to the appearance of a minimal time of controllability~\cite{Loheac2017, Pighin2018, Loheac2018, Loheac2021}.

In the case of unilateral constraints for linear control problems in finite dimension, these obstructions can be categorised thanks to Brunovsky's normal form as done in \cite{Loheac2018}, leading to the existence of a positive minimal time. In infinite dimension, however, we are only aware of obstructions based on the comparison principle (see \cite{Loheac2017} and~\cite{Pouchol2019}). The present work uncovers another type of obstruction, already hinted at in~\cite{Pighin2018}.

\subsection{Main results}\label{subsec-results}

As our results require different sets of hypotheses and in order to give a quick glance at the main ideas, we first present them in the simplified context of the heat equation~\eqref{nonnegative-controllability-heat}. 

Given a constraint set $\U_+ \subset L^2(\Omega)$, we will be considering control constraints of the form \[\forall t \in (0,T), \quad u(t) \in \U_+. \] 
Here, the notation $\U_+$ emphasises that we will always deal with constraints that include the nonnegativity constraint, \ie sets $\U_+$ such that $\U_+ \subset \{u \in L^2(\Omega), \,u \geq 0\}$.

Now, when the control $u$ satisfies $u \geq 0$, it follows from \emph{the parabolic comparison principle} satisfied by the Dirichlet Laplacian~\cite{Evans1998} that
\begin{equation}
    \label{comp-heat}
\forall t \geq 0, \quad y(t) \geq e^{t\Delta} y_0,
\end{equation}
where $(e^{t\Delta})_{t \geq 0}$ denotes the heat semigroup with Dirichlet boundary conditions. 
Hence, targets~$\tar$ which do not satisfy $\tar \geq e^{T\Delta} y_0$ cannot be reached with nonnegative controls, let alone on-off shape controls.

%and we will sometimes abusively write $u \in \U_+$ to refer to the condition $u(t) \in \U_+$ for all $t \in (0,T)$.

Taking into account the obstruction to controllability given by the inequality~\eqref{comp-heat}, we adapt the usual definition of approximate controllability to the context of nonnegative controls.

More precisely, we say that system \eqref{nonnegative-controllability-heat} is \textit{nonnegatively approximately controllable  with controls in $\U_+$ in time $T>0$}, if for all $\e>0$, and all $y_0, \tar \in L^2(\Omega)$ such that $\tar \geq e^{T \Delta} y_0$, there exists a control $u \in L^2((0,T) \times \Omega)$ with values in $\U_+$ such that the corresponding solution to \eqref{nonnegative-controllability-heat} satisfies $\|y(T) - \tar \|_{L^2(\Omega)} \leq \e$.

\paragraph{On-off shape control.}

For our first main result, we focus on nonnegative approximate controllability with on-off shape controls: for a fixed $L\in(0,1)$, we consider the constraint set 
\[\U_L^{\mathrm{shape}}:=\{M \chi_\omega, \quad \omega\subset \Omega, \quad |\omega|\leq \mL, \, M > 0\} \subset L^2(\Omega).\]

%When considering shape controls, we thus depart from constraining controls to a fixed subset $\omega$, but still require that at all times, the control act only on part of the domain: for a given fixed $L\in (0,1)$, we enforce 
%\[ |\mathrm{supp}(u)| \leq \mL.\]
Within the above class of on-off shape controls, we establish nonnegative approximate controllability in arbitrary time (see Theorem \ref{thm-controllability} for the precise and general statement), whatever the value of $L \in (0,1)$.
\begin{informalth}
\label{informal-shape}
For any $L\in (0,1)$, $T>0$, system~\eqref{nonnegative-controllability-heat} is nonnegatively approximately controllable with controls in $\U_{\mathrm{shape}}^L$ in time~$T$.
\end{informalth}

To establish this result, we draw from the Lions strategy in \cite{Lions-1992}, which develops a constructive approach in studying the approximate controllability of a linear wave equation.
The idea is to consider the requirement $\|y(T)-\tar\|_{L^2(\Omega)} \leq \e$ as a constraint. With $L_T u :=\textstyle \int_0^T e^{(T-t)\Delta} u(t) \,dt$ and since $y(T)= L_T u + e^{T \Delta} y_0$, Lions considers the constrained optimal control problem
\[\pi:=\inf \left\{\frac{1}{2} \|u\|_{L^2((0,T) \times \Omega)}^2, \ \|e^{T \Delta} y_0 + L_T u -\tar\|_{L^2(\Omega)}\leq \e \right\}.\]
The infimum satisfies $\pi<+\infty$ if and only if there exists $u \in L^2((0,T)\times \Omega)$ steering $y_0$ to a closed $\e$-ball around~$\tar$.
To find minimisers, \ie to build controls, note that 
\[\pi=\inf_{u\in L^2((0,T) \times \Omega)}\textstyle \frac{1}{2} \|u\|_{L^2((0,T) \times \Omega)}^2 + G_{T,\e}(L_T u)= \displaystyle \inf_{u\in L^2((0,T) \times \Omega)} F_T(u) + G_{T,\e}(L_T u),\]
with $F_T(u) = \textstyle \frac{1}{2} \|u\|_{L^2((0,T) \times \Omega)}^2$ and
\[G_{T,\e}(y)=\left\{\begin{aligned}  0 \quad &  \text{if} \quad \|e^{T\Delta} y_0+y-\tar\|_{L^2(\Omega)}\leq \e, \\
 +\infty \quad & \textrm{otherwise}.\end{aligned}\right.\]
From this optimisation problem, one computes its Fenchel dual optimisation problem, which reads 
\[d:=-\inf_{\adjtar \in L^2(\Omega)} F_T^\ast(L_T^\ast \adjtar) + G_{T,\e}^\ast(-\adjtar) = - \inf_{\adjtar \in L^2(\Omega)} \textstyle \frac{1}{2} \|L_T^\ast \adjtar\|_{L^2((0,T) \times \Omega)}^2 + G_{T,\e}^\ast(-\adjtar),\]
where $\cost^\ast$($= \cost$) and $G_{T,\e}^\ast$ are the Fenchel conjugates of $\cost$ and $G_{T,\e}$, respectively, and $L_T^\ast$ is the adjoint of the linear bounded operator $L_T : L^2((0,T)\times\Omega) \rightarrow L^2(\Omega)$. Recall that for a given $\adjtar \in L^2(\Omega)$, $p = L_T^\ast \adjtar$ is the solution to the adjoint equation ending at $\adjtar$, \ie it solves
\begin{equation}
\label{backward-heat}
\left\{\begin{aligned}
    p_t+\Delta p &= 0, \\
    p=0 \ &\textrm{on} \ \partial \Omega, \\
    p(T) = \adjtar  \ &\textrm{on} \  \Omega.
    \end{aligned}\right.
\end{equation}

Under suitable conditions, the Fenchel-Rockafellar theorem \cite{Rockafellar1967} ensures that $\pi =d$.
As a result, one can then study the dual functional to establish that $\pi=d<+\infty$, and that its minimum is attained. Typically, one proves that it is coercive, as a consequence of a unique continuation property. Furthermore, the  cost function $\cost$ is differentiable in this case and the first order optimality condition for the (unique) variable $\adjtar^\star$ minimising the dual functional then reads~$L_T L_T^\ast \adjtar^\star=\tar - e^{T \Delta} y_0-\e \textstyle\frac{\adjtar^\star}{\|\adjtar^\star\|_{L^2(\Omega)}}$. The optimal control $u^\star:=L_T^\ast \adjtar^\star$ is thus constructed from the minimiser of the dual problem~$\adjtar^\star$.

Accordingly, in this paper we reframe constrained approximate controllability as an optimal control problem, replacing $\textstyle \frac{1}{2} \|u\|_{L^2((0,T) \times \Omega)}^2$ of~\cite{Lions-1992} with a suitable cost functional $\cost$. This constitutes a novel generalisation of the Lions method.

One can choose between two different sufficient conditions for the equality $\pi=d$ to hold. One regards the primal problem, and the other the dual problem. Importantly, they are not symmetric (although the primal and dual problems are). These hypotheses when used on the primal problem are useless when it comes to proving controllability: they amount to assuming that controllability holds. We here crucially use these hypotheses in terms of the dual problem, see subsection \ref{subsec-abstract-control} and Appendix~\ref{app-subsec-fenchel} for more details.

As the detailed statements in Theorem \ref{thm-controllability} and Proposition~\ref{prop-second-cost} show:
\begin{itemize} 
\item Instead of using the $L^2$ norm as in the optimal control problem studied in \cite{Lions-1992}, we will consider the cost functional:
\begin{equation}\label{intro-cost-function}
\cost(u) :=\frac12 \sup_{t \in [0,T]}  \max\left(\|u(t,\cdot)\|_{L^\infty(\Omega)}, \frac{\|u(t,\cdot)\|_{L^1(\Omega)}}{\mL}\right)^2 + \delta_{\{u \geq 0\}}(u),\end{equation}
where $\delta_{\{u \geq 0\}}(u) = 0$ if $u \geq 0$ and $+\infty$ otherwise, and the supremum over $t \in [0,T]$ is the essential supremum.

The rather unusual form of the minimisation criterion \eqref{intro-cost-function} is finely designed so as to handle nonnegativity and the other (bound, volume) constraints we are dealing with. 
\item The optimal controls have constant amplitude in time, \ie $M(t) \equiv \overline{M}$.
\item The proof is constructive: 
%the computations of \cite{Lions-1992} 
the optimal control $u^\star$ can be computed from a unique dual optimal variable $\adjtar^\star$ solving the corresponding Fenchel dual problem. This computation generalises what is done in \cite{Lions-1992} to the broader case of costs that are not differentiable but still convex. More precisely, $u^\star$ is given 
by 
\begin{equation*} 
u^\star(t, \cdot) = \overline{M} \, \chi_{\{p^\star(t, \cdot) > h(p^\star(t, \cdot))\}},\quad \overline{M} = \int_0^T \int_{\{p^\star(t, \cdot) > h(p^\star(t,\cdot))\}} p^\star(t,x)\,dx \,dt, \end{equation*}
where $h : L^2(\Omega) \rightarrow \R$ is a function that will be defined in Section \ref{subsec-bathtub-convex-analysis}, and $p^\star$ solves the adjoint equation~\eqref{backward-heat} with $p^\star(T)=\adjtar^\star$.

\end{itemize}

\paragraph{Obstructions to nonnegative controllability.}
In the spirit of the unconstrained case, one may wonder whether nonnegative approximate controllability can be achieved with controls acting only in some prescribed time-independent subdomain $\omega$. We emphasise that our first result does not \textit{a priori} prevent the control from visiting the whole domain $\Omega$.

Our second result proves that visiting the whole $\Omega$ is necessary in the following sense: if the sets $\omega(t), \, t\in (0,T)$ do not intersect some fixed open subset of $\Omega$, nonnegative approximate controllability is lost for small times.
%Once the obstruction due to the comparison principle has been taken into account, one may hope to obtain nonnegative approximate controllability with controls acting only in some prescribed time-independent subdomain $\omega$, in the spirit of the unconstrained case.  Our first main result is, perhaps surprisingly, a negative answer for small times $T>0$ (see Theorem \ref{thm-obstruction} for the precise statement). 
\begin{informalth}
\label{informal-obstruction}
Assume that the constraint set $\U_+$ satisfies the following property: there exists a ball $B(x,r)\subset \Omega$ with $x \in \Omega$ and $r>0$ such that
\[\forall u \in \U_+, \quad \mathrm{supp}(u) \cap B(x, r) = \emptyset.\]
Then, there exists $T^\star>0$ such that the control system~\eqref{nonnegative-controllability-heat} is not nonnegatively approximately controllable with controls in $\U_+$ in time~$T \leq T^\star$.
\end{informalth}
We refer to Theorem \ref{thm-obstruction} for the complete statement. Let us mention that obstructions of this type have been reported for similar problems in~\cite{Pighin2018}.

%In other words and informally speaking, nonnegative controls must be allowed to act everywhere for nonnegative controllability in arbitrary time $T>0$ to hold true.  

\paragraph{Amplitude and time optimal control.}
In Section~\ref{sec-comments}, we gather several further results regarding the dependence of the amplitude $\overline{M} = \overline{M}(T, y_0, \tar, \e)$ on its arguments. Using duality once more, we study its dependence on the final time $T$. 

Focusing on the case $y_0 = 0$, we then
establish an equivalence between the optimal control problem and the related minimal time problem
\begin{equation*}
    %\label{intro-time-optimal-control-problem}
  \inf\{T>0, \quad \exists u\in L^2((0,T) \times \Omega),
  \quad \|L_T u-\tar\|_{L^2(\Omega)} \leq \e, \quad \cost(u)\leq \lambda\}, \quad \lambda>0.
\end{equation*}
%We then investigate the dependence of the amplitude $\overline{M} = \overline{M}(T, y_0, \tar, \e)$ with respect to its arguments, more specifically how it depends on the final time $T$, using the dual problem and generic properties of parametric optimisation problems. In particular, if the control system is exponentially stable, $T\mapsto \overline{M}(T, y_0, \tar, \e)$ remains bounded away from 0, as expected with bounded controls (see Proposition~\ref{prop-bound_amplitude} for the precise statement).

%We then focus on the case
%\[y_0=0,\]
%and establish an equivalence between the optimal control problem and the related minimal time problem
%\begin{equation}
%    \label{intro-time-optimal-control-problem}
%  \inf\{T>0, \quad \exists u\in L^2((0,T) \times \Omega),
%  \quad \|L_T u-\tar\|_{L^2(\Omega)} \leq \e, \quad \cost(u)\leq \lambda\}, \quad \lambda>0,
%\end{equation}
%in the spirit of \cite{wang_equivalence_2012, Kunisch-Wang-2013}. {\color{black}We point out that in these references, the bound on the control is given in terms of $L^p$ norms, which is here replaced by our unusual cost functional $\cost$.} This allows us to study the minimal time control quite comprehensively (see Theorem \ref{thm-existence-time-optimal-controls} for the precise statement).

\subsection{General results}
\label{subsec-general-results}
Theorems \ref{informal-shape} and~\ref{informal-obstruction} above have been stated for the heat equation with Dirichlet boundary conditions, in order to provide the reader with a quick overview of our main results. In fact, they all hold for more general semigroups under suitable hypotheses presented hereafter.  

The underlying general setting is that of linear control problems of the form 
\begin{equation}\label{nonnegative-controllability}
\left\{\begin{aligned}
    y_t - A y &= u, \\
    %y=0 \ &\textrm{on} \ \partial \Omega, \\
    y(0)& =  y_0  \ \textrm{in} \ \Omega
    \end{aligned}\right.
\end{equation}
where $\Omega$ is an open subset of $\R^d$, and $A: D(A) \rightarrow L^2(\Omega)$ is an operator generating a $C_0$ semigroup $(S_t)_{t \geq 0}$ on~$L^2(\Omega)$~\cite{Engel_Nagel_2001, pazy2012semigroups}.

In this more general context, we define nonnegative approximate controllability as follows. 
\begin{dfntn}
Given a constraint set of \textbf{nonnegative} controls $\mathcal  U_+ \subset L^2(\Omega)$, we say that system~\eqref{nonnegative-controllability} is nonnegatively approximately controllable with controls in $\U_+$ in time $T$ if for all $\e>0$, and all $y_0, \tar \in L^2(\Omega)$ such that $\tar \geq S_T y_0$, there exists a control $u \in L^2((0,T) \times \Omega)$ with values in $\U_+$ such that the corresponding solution to \eqref{nonnegative-controllability} satisfies $\|y(T) - \tar \|_{L^2(\Omega)} \leq \e$.
\end{dfntn}

\paragraph{General hypotheses for Theorem \ref{informal-shape}.}
We have previously presented Theorem \ref{informal-shape} for the heat equation as a paradigmatic example.
%a paradigmatic example of a classical system for which our results apply, and of potential applications.
Nevertheless, the underlying hypotheses on which some of our proofs rely are much more general in nature; we review them below.
%Let us introduce the key properties about the operator $A$ which we will need to state some of our results:
%We will say that a given operator $(A, D(A))$  generating a $C_0$ semigroup denoted $(S_t)_{t \geq 0}$

\begin{itemize} 
    \item First, we consider the (unusual) unique-continuation like property
    \begin{equation}
    \label{main_property}\tag{\textbf{GUC}}
  \forall y \in L^2(\Omega), \quad \exists t \in (0,T), \; S_t y\, \text{ is constant over }\, \Omega \quad  \implies \quad y = 0.
\end{equation}
This property is satisfied as soon as the three assumptions below hold:
\begin{itemize}
\item for $y \in L^2(\Omega)$, $S_t y  \in D(A)$ for all $t>0$ (for instance, this is true if $(S_t)_{t \geq 0}$ is analytic~\cite{pazy2012semigroups}),
\item the only constant function in $D(A)$ is the zero function,\footnote{This is the case for the Dirichlet Laplacian with domain $D(A) = H^2(\Omega) \cap H^1_0(\Omega)$ if $\Omega$ has a $C^2$ boundary.}
\item $S_t$ is injective for all $t > 0$.\footnote{This is the case for groups, such as the wave equation, and for parabolic equations thanks to the parabolic maximum principle. This is also true for analytic semigroups: if $S_t y = 0$ for some $t>0$, then $S_s y = 0$ for all $s \geq t$ and  by analyticity $S_s y = 0$ for all $s \geq 0$, which for $s=0$ yields $y=0$.}
\end{itemize}
\item Second, we will be interested in \textit{analytic-hypoellipticity}: $\partial_t -A$ is said to be analytic-hypoelliptic
     if
    any distributional solution $y$ to $\partial_t y - A y = f$ on $\Omega \times (0,T)$ with $f$ analytic in $\Omega$ is analytic in $\Omega$, where analyticity refers to real-analyticity.
\item Third, we will say that $(S_t)_{t \geq 0}$ satisfies the \textit{comparison principle} if
\begin{equation}
\label{comp_principle}
\forall y \in L^2(\Omega), \quad y \geq 0 \implies \forall t >0, \; S_t y \geq 0.
\end{equation}
\end{itemize}

%The first property is sufficient for the generalisation of Theorem~\ref{informal-obstruction}, see Theorem~\ref{thm-obstruction}. 
The first two properties are sufficient for the generalisation of Theorem~\ref{informal-shape}, see Theorem~\ref{thm-controllability}. The third will play an important role when it comes to minimal controllability times, and is in line with our definition of nonnegative approximate controllability.

\paragraph{Elliptic operators.}
As a generalisation of the Dirichlet Laplacian, let us discuss a large class of uniformly elliptic operators that do satisfy these properties and to which our obstruction result Theorem~\ref{informal-obstruction} generalises (see Theorem~\ref{thm-obstruction}). 

Let us assume that $\Omega$ is a bounded, open, connected subset of $\R^d$, with $C^2$ boundary. Defining $D(A):= H^1_0(\Omega) \cap H^2(\Omega)$, we introduce operators of the form
\begin{equation}
\label{def_elliptic}
\forall y \in D(A), \quad A y:= \sum_{1\leq i,j \leq d} \partial_{x_j} (a_{ij}(x) \partial_{x_i}y) - \sum_{i=1}^d b_i(x) \partial_{x_i} y + c(x) y.
\end{equation}
When referring to operators of the form~\eqref{def_elliptic}, we will always assume that the functions $a_{ij} = a_{ji}$, $b_i$ are in $W^{1,\infty}(\Omega)$, $c$ is in $L^\infty(\Omega)$, 
and that the operator is uniformly elliptic, \ie there exists $\theta >0$ such that
\[\forall x \in \Omega,\; \forall \xi \in \R^d,  \quad \sum_{1 \leq i,j \leq d} a_{ij}(x)\xi_i \xi_j \geq \theta |\xi|^2. \]
The adjoint of $A$ is given by
\[\forall p \in D(A^\ast), \quad A^\ast p= \sum_{1\leq i,j \leq d} \partial_{x_i} (a_{ij}(x) \partial_{x_j}p) + \sum_{i=1}^d b_i(x) \partial_{x_i} p + \left(c(x)- \sum_{i=1}^d \partial_{x_i}(b_i(x))\right) p,\]
and we have $D(A^\ast) = D(A)$.

Both $A$ and $A^\ast$ satisfy the parabolic comparison principle~\cite{Evans1998}, hence they satisfy the comparison principle~\eqref{comp_principle}.  
They also satisfy the three conditions sufficient for the~\eqref{main_property} property to hold.~\footnote{The analyticity of the semigroup is well known for this class of elliptic operators on open domains with $C^2$ boundary. There are clearly no nonzero constant functions in $H^2\cap H^1_0$. Finally, injectivity follows from the comparison principle (see above footnote).}
Furthermore, both $\partial_t-A$ and $\partial_t-A^\ast$ are analytic-hypoelliptic as soon as all functions $a_{ij}$, $b_i$ and~$c$ are analytic~\cite{Nelson1959}. 

%Finally, note that the Laplace operator with Dirichlet boundary conditions, introduced at the beginning of the article, is a particular case of a uniformly elliptic operator with analytic coefficients.

\subsection{Proof strategy and related works}

 In the unconstrained case, approximate controllability of the heat equation is a consequence of the unique continuation property, thanks to a general property of linear control problems (see for example \cite[Section 2.3]{CoronBook}).
In the case of heat equations, the latter property can be obtained by the Holmgren Uniqueness Theorem~\cite{aronszajn1956unique}.
In contrast to these existence results, the variational approach developed in~\cite{Lions-1992} (see Section \ref{subsec-results}), handles approximate controllability in a \emph{constructive manner}.

Our strategy consists in extending this approach to the constrained case: the main idea is to find a suitable cost function $\cost$ such that optimal controls must satisfy the constraint $u \in \U_L^{\mathrm{shape}}$. A remarkable feature of our strategy lies in how we design the cost function: we do so by building an adequate Fenchel dual function, instead of trying to find the cost function directly.

\paragraph{Constrained controllability.}
Constrained control problems in infinite dimension have been studied in papers such as~\cite{Kunisch-Wang-2013, Berrahmoune2014, Berrahmoune2019, berrahmoune_variational_2020, Ervedoza_2020}.
In~\cite{Ervedoza_2020}, sufficient conditions (in the form of unique continuation properties) for controllability results are derived when the control and states are constrained to some prescribed subspaces, but at the expense of controlling only a finite-dimensional subpart of the final state. 
In \cite{Kunisch-Wang-2013}, the authors deal with a form of approximate controllability of the heat equation akin to ours, focusing on minimal time problems. They derive bang-bang type necessary optimality conditions for minimal time controls, and then build such controls using an auxiliary optimisation problem.  
%In \cite{Berrahmoune2014}, exact controllability with constraints is considered. Abstract necessary and/or sufficient conditions are derived, in the form of modified observability inequalities, but the applicability of the method is demonstrated on a few examples with isotropic constraints.
%Additionally, in these references, controls are studied by introducing so-called dual functionals, drawing from the variational formulation of the Hilbert Uniqueness Method. 

The papers \cite{Berrahmoune2014, Berrahmoune2019, berrahmoune_variational_2020} address constrained exact controllability through modified observability inequalities, thus giving abstract necessary and/or sufficient conditions. One key difference with our work is that constraint sets are assumed to be convex. In fact, all examples handled by~\cite{Berrahmoune2014, Berrahmoune2019, berrahmoune_variational_2020} feature isotropic constraints, that is, constraints that are symmetrical with respect to $0$, or more generally, are expressed using radial functions (such as norms). This precludes, for instance, any type of positivity constraint.
%

%In \cite{Berrahmoune2019, berrahmoune_variational_2020},

It is noteworthy that all the above references introduce so-called dual functionals, drawing from the variational formulation of the Hilbert Uniqueness Method. However, the formalism of Fenchel-Rockafellar duality in itself, as developed in \cite{Lions-1992}, has increasingly been abandoned in the literature. Some notable exceptions are~\cite{trelat-wang-xu} in the context of stabilisation and~\cite{Lazar2023} for parameterised problems, both in the unconstrained case. To some extent, the work~\cite{Berrahmoune2019} uses Fenchel duality to study (constrained) null-controllability in some specific settings.

We fully exploit the ideas hinted at in the latter paper by choosing a different type of functional, which allows us to handle anisotropic, non-convex constraints.
In contrast with the aforementioned trend in the literature, we work with Fenchel duality, but in a rather unusual way, in that we will focus mainly on the dual problem. The nature of the actual primal problem (optimal control problem) being solved follows effortlessly.
To perform the necessary computations, we will make extensive use of convex analysis. Doing so bypasses many technical difficulties thanks to properties of subdifferentials and Fenchel conjugates, among others, and allows for the use of costs which are not differentiable but still smooth in the convex analytic sense.

\paragraph{Bathtub principle for appropriate costs.}
The second main idea is what underlies our choice of cost function~$\cost$, forcing optimal controls to satisfy the required the on-off shape constraint.
As the set of on-off shape controls is a non-convex cone, we are led to relaxation, \ie to consider the closure of its convex hull. In order to build relevant costs, we then rely on the so-called bathtub principle (actually, a relaxed version of it)~\cite{Lieb2001}.

For a given function $v \in L^2(\Omega)$, the latter principle solves
\begin{equation*}
%\label{relaxed-bathtub-intro}
\sup_{u \in \overline{\U}_L} \int_\Omega u(x) v(x) \,dx, \qquad \overline{\U}_L:= \left\{ u \in L^2(\Omega), \; 0 \leq u \leq 1  \text{ and } \int_\Omega  u \leq \mL\right\}.\end{equation*}
This optimisation problem comes up naturally in some control problems similar to ours~\cite{lance_shape_2020, mazari_quantitative_2021}, or in shape optimisation problems~\cite{privat_optimal_2015}.

Interpreting the bathtub principle as a Fenchel conjugate leads us to design the unusual cost functional~\eqref{intro-cost-function}.
This allows us to design dual problems such that optimal controls exist, and are characterised as maximisers of some bathtub principle. Then, using analyticity properties for solutions of the dual problem, we prove their uniqueness and hence their extremality, thereby uncovering that they are on-off shape controls.

\paragraph{Bang-bang property of optimal controls.}
Bang-bang controls (\ie controls that saturate their constraints) are a common feature in time optimal control problems. A growing literature on the heat equation alone \cite{schmidt1980bang, mizel1997bang,micu_time_2012,wang_attainable_2015, yang_bangbang_2019} shows that this property extends well to some infinite-dimensional systems. In our case, we will see that the on-off  shape controls we have constructed can be understood as time-optimal controls. As these controls are bang-bang, this yields another occurrence of the bang-bang property in the time-optimal control of the heat equation. 

Note, however, that in the references cited above, the controls are constrained to lie in balls of specific function spaces, whereas we consider non-negative constraints on the controls, which is an anisotropic constraint. Moreover, the bang-bang property is usually established separately using optimality conditions, having established controllability at the onset. In our case, the Fenchel-Rockafellar duality approach allows to do all those things simultaneously.

\subsection{Extensions and perspectives}

\paragraph{Operator, boundary conditions.}
The~\eqref{main_property} property and analytic-hypoellipticity are two key sufficient properties for nonnegative approximate controllability by on-off shape controls. We have highlighted second-order elliptic operators with analytic coefficients Dirichlet boundary conditions as an example. Our results apply to such operators with Robin boundary conditions of the form $a(x) y + b(x) \partial_\nu y = 0$ over $\partial \Omega$ (with $a, b$ analytic) as soon as the function $a$ does not vanish on the whole of $\partial \Omega$ (more generally, as soon as $a$ is nontrivial on any connected component of $\partial \Omega$).
This excludes the important case of Neumann boundary conditions, which remains open. 

Our approach also accommodates subelliptic operators. This includes a large class of Hörmander operators, \ie operators of the form $A = \textstyle \sum_{i=1}^m X_i^2 + X_0 + V \mathrm{Id}$ with vector fields $X_1, \ldots, X_m$ generating a Lie algebra that equals $\R^d$ on the whole of $\Omega$. Under general regularity assumptions and boundary conditions, such an operator and its adjoint generate a strongly continuous semigroup on $L^2(\Omega)$, satisfy the comparison principle~\cite{Bony69}, all three conditions sufficient for the~\eqref{main_property} property, and are analytic-hypoelliptic for instance if the characteristic manifold is an analytic symplectic manifold (see~\cite{Metivier1981}).

Finally, going beyond the linear setting is a completely open problem, since our approach fundamentally relies on the Fenchel-Rockafellar theorem which itself requires a bounded linear operator (the role played by $L_T$ in our setting).
%For instance, we do not know how to tackle the case of $A$ being some maximal monotone operator.}

\paragraph{Control operator.}
Our results have been stated with the identity control operator. 
They extend to the nonnegative control of 
\begin{equation*}
\left\{\begin{aligned}
    y_t - A y &= \varphi u, \\
    %y=0 \ &\textrm{on} \ \partial \Omega, \\
    y(0) = &y_0  \ \textrm{in} \ \Omega
    \end{aligned}\right.
\end{equation*}
where $\varphi \in L^\infty(\Omega)$ is positive, analytic.

An interesting perspective is to follow our proof strategy with boundary control operators, where on-off shape controls now refer to characteristic functions over the boundary $\partial \Omega$.    

\paragraph{Other notions of controllability.}
In the case of unconstrained controllability with a control acting in some fixed subset $\omega$, any function that can be reached exactly is (at least) analytic in $\Omega \setminus \omega$, preventing exact controllability to hold true. 

On the one hand, this argument for (non)-exact controllability by on-off shape controls fails since the control may act everywhere. On the other hand, our approach heavily relies on targeting a ball $\overline{B}(\tar, \e)$ with~$\e>0$. As a result, exact nonnegative controllability by on-off shape controls is an open and seemingly difficult question.

A related matter is that of the cost of approximate controllability as a function of $\e \rightarrow 0$.

Although our focus has been on $L^2$-approximate controllability, we mention that one may extend the same methodology to $L^p$-approximate controllability for $1<p<+\infty$, by working in duality within the appropriate spaces:
%by working with the appropriate pairs of spaces: 
the bounded operator underlying the Fenchel-Rockafellar duality is now $L_T \in L(L^2(0,T;L^p(\Omega)),L^p(\Omega))$, meaning that the dual functional is defined on $L^{q}(\Omega)$ with $q$ the dual exponent to $p$.

\paragraph{Controllability in large time.} As evidenced by Theorem B, we provide obstructions for small times~$T$. We do not know whether nonnegative approximate controllability holds for sufficiently large times.

\paragraph{Abstract constrained control.}
The strategy of proof developed in this article hints at generalisations, where the method is applied to abstract linear control problems with abstract constraint sets $\U$. 

In particular, we expect it to lead to necessary and  sufficient conditions for controllability when $\U$ is convex. When~$\U$ is not convex as is the case for on-off shape controls, this requires to study the convex hull of $\U$, following the relaxation approach. This abstract setting should allow us to discern how one can design a cost function $\cost$, analogous to \eqref{intro-cost-function}, tailored to a given $\U$.

Further sufficient conditions should be derived to ensure that optimal controls in the convex hull of~$\U$ actually are in the original constraint set~$\U$. In the present work, analytic-hypoellipticity and the~\eqref{main_property} property play that role in the case of on-off shape controls.
%, further exploiting the power of convex analysis and its geometrical aspects to paint a broader picture of constrained controllability with more general constraint sets. 

This will be the subject of an ulterior article. 
 
\paragraph{Regularity of the sets $\omega(t)$.}
Another problem is to analyse the complexity of the sets $\omega(t)$ occupied by optimal controls over time. For instance, how smooth ($BV$~regularity, number of connected components, etc) are the sets $\omega(t)$ achieving approximate controllability?

In view of applications, these are important issues for the controls to be implementable in practice. For example, if the sets $\omega(t)$ are constrained to depend on a few parameters restricting their geometry, or if they are restricted to rigid movements, controllability is a totally open question.

\paragraph{Homegenisation approach.}
We acknowledge that an homogenisation approach to establishing  nonnegative approximate controllability by on-off shape controls could certainly be pursued. The underlying idea would be to "atomise" the sets $\omega(t)$ (see \cite{allaire1997shape}). 
Contrarily to our technique, however, this approach would not be constructive.

\paragraph{Numerical approximation of optimal controls.}
Optimal controls are given explicitly in terms of optimisers of the dual problem: the constructive nature of our approach means that optimal controls may be numerically computed, at least on paper.

Providing reliable and efficient methods to compute optimal controls is a difficult issue which has been studied in the case of Lions' cost functional with $\e = 0$ (\ie exact controllability)~\cite{Labbe-Trelat-2006, Boyer2013}. Similar results in a generalised setting with our Fenchel-Rockafellar-based approach would be valuable.

Contrary to Lions' cost functional, we note that ad hoc algorithms are required in order to cope with functions that are not necessarily differentiable, as is the case in the present paper. Recent primal-dual algorithms designed for optimisation problems with objective functions of the form $F(u) + G(L_T u)$ are likely to be good candidates~\cite{Chambolle_2011}.

\paragraph{Outline of the paper.}
First, Section~\ref{sec-choosing-cost} lays out the convex analytic framework, that of Fenchel-Rockafellar duality, and how it may be applied to constrained approximate controllability. We then introduce the bathtub principle and interpret it in terms of Fenchel conjugation in order to design a relevant optimal control problem for our purposes.
%, showing it is a natural tool in deriving relevant costs for the problem at hand.  
Section~\ref{sec-controllability} is dedicated to the proof of our nonnegative approximate controllability result given by Theorem~\ref{thm-controllability}, and Section~\ref{sec-obstructions} to that of the  obstruction result, Theorem~\ref{thm-obstruction}.
Finally, Section~\ref{sec-comments} gathers our results about further obstructions when the control amplitude is bounded, along with our analysis of the corresponding minimal time control problem.

\section{Building the optimal control problem}
\label{sec-choosing-cost}
\subsection{Convex analytic framework}
\label{subsec-convex-analysis}
Let $H$ be a Hilbert space. We let $\Gamma_0(H)$ be the set of functions from $H$ to $\mathopen{]}-\infty, +\infty]$ that are convex, lower semicontinuous (abbreviated lsc) and proper (\ie not identically $+\infty$).  
For $f \in \Gamma_0(H)$, we let \[\mathrm{dom}(f) = \{x \in H,\, f(x) < +\infty\}\] be its \textit{domain}.

\paragraph{Fenchel conjugate.} For a proper function $f : H \rightarrow  \mathopen{]}-\infty, +\infty]$, we denote $f^\ast : H \rightarrow \mathopen{]}-\infty, +\infty]$ its convex conjugate, given by the convex lsc function
\[f^\ast(y): = \sup_{x \in H} \; \big(\langle y, x \rangle - f(x)\big), \quad \forall y \in H.\]

\paragraph{Support and indicator functions.} 
Given a subset $C \subset H$, the \textit{indicator function} of $C$ is the function defined by
\begin{equation*}
\delta_C(x) := 
    \begin{cases}
    0 & \text{if } x \in C \\
    +\infty & \text{if } x \notin C 
    \end{cases}\;, \quad \forall x \in H,
\end{equation*}
and the \textit{support function} of $C$ is defined by
\[\sigma_C(p):=\sup_{x\in C}\langle p, x \rangle =\delta_C^\ast(p), \quad \forall p \in H,\]
\textit{i.e.}, the Fenchel conjugate function of the indicator function of $C$.

\paragraph{Subdifferentials.} For $f \in \Gamma_0(H)$, we let 
\[\partial f(x) := \{p \in H, \, \forall y \in H, f(y) \geq f(x) + \langle p, y-x\rangle\},\]
be its subdifferential at a point $x \in H$.

Various common properties of Fenchel conjugates, support functions and subdifferentials are used throughout the article. These are all recalled in Appendix~\ref{app-convex-analysis}, where a few additional lemmas are proved.

\subsection{Approximate controllability by Fenchel duality (\cite{Lions-1992})}
\label{subsec-abstract-control}

Let us explain how the approximate controllability problem is reformulated in the context of Fenchel-Rockafellar duality \cite{Rockafellar1967} (see \ref{app-subsec-fenchel} for a general presentation), following the strategy introduced by Lions in~\cite{Lions-1992}. We work with the control problem~\eqref{nonnegative-controllability}, with the control space $E := L^2((0,T) \times \Omega)$ and the state space $L^2(\Omega)$. %From now on, the norm $\|\cdot\|$ and inner product $\langle\cdot, \cdot\rangle$ will refer to the standard $L^2(\Omega)$-norm and inner product. The norms $\|\cdot\|_1$ and $\|\cdot\|_\infty$ will refer to the norms of $L^1(\Omega)$ and $L^\infty(\Omega)$, respectively.

By Duhamel's formula $y(T) = S_T y_0 + L_T u$, the inclusion $y(T) \in \overline{B}(\tar, \e)$ (where the closed ball of center $\tar$ and radius $\e$ is with respect to the $L^2(\Omega)$-norm) can equivalently be written as $L_T u \in \overline{B}(\tar - S_T y_0, \e)$.
%with the notation $\tar - S_T y_0 := \tar - S_T y_0$.  

Given some cost functional $\cost : E \rightarrow [0, +\infty] \in \Gamma_0(E)$, consider the optimal control problem (which we will refer to as the \textit{primal} problem)
\begin{equation*}
\pi:= \inf_{u \in E} F_T(u) + G_{T,\e}(L_T u).
\end{equation*}
where \[G_{T,\e} := \delta_{\overline{B}(\tar - S_T y_0, \e)} \in \Gamma_0(L^2(\Omega)). \]
%is a function in Note that $G \in \Gamma_0(L^2(\Omega))$. Here, note that $\pi \geq 0$. 
%The above choice of optimal control problem leads to the obvious result 
%\[\text{approximate controllability with controls in $\U_2$, in time $T$} \; \iff \; \forall \e>0, \forall \tar \geq S_T y_0, \;  \pi \text{ is finite}.\]

%Going further, we have
%\[ \forall \e>0, \forall \tar \geq S_T y_0, \;  \pi \text{ attained by some $u \in \U_{1}$ }\; \implies \; \text{approximate controllability with controls in $\U_{1}$, in time $T$}. \]

Now consider the Fenchel \textit{dual} to the above problem, which writes 
\begin{equation}\label{dual-control-problem}
    d = -\inf_{\adjtar \in L^2(\Omega)} J_{T,\e}(\adjtar), \quad J_{T,\e}(\adjtar):= F_{T}^\ast(L_T^\ast \adjtar) + G_{T,\e}^\ast(-\adjtar).
\end{equation}
Thanks to the formulae for conjugates, we find 
\[G_{T,\e}^\ast(z)= \langle \tar - S_T y_0, z\rangle_{L^2} + \e \|z\|_{L^2}, \]leading to
\[J_{T,\e}(\adjtar) =  F_{T}^\ast(L_T^\ast \adjtar) - \langle \tar - S_T y_0, \adjtar\rangle_{L^2}+ \e \|\adjtar\|_{L^2}.\]

We recall that $p = L_T^\ast \adjtar$ solves the adjoint equation 
\begin{equation}
\label{backward}
\left\{\begin{aligned}
    p_t+A^\ast p &= 0, \\
    %p=0 \ &\textrm{on} \ \partial \Omega, \\
    p(T) = \adjtar  \ &\textrm{in} \  \Omega.
    \end{aligned}\right.
\end{equation}

\paragraph{Strong duality.}  The weak duality $\pi \geq d$ always holds. According to the Fenchel-Rockafellar duality theorem recalled in Appendix~\ref{app-subsec-fenchel}, the existence of $\adjtar \in \mathrm{dom}(G_{T,\e}^\ast)$ such that $F_{T}^\ast$ is continuous at $L_T^\ast \adjtar$ is a sufficient condition for the strong duality $\pi = d$ to hold. Since $\mathrm{dom}(G_{T,\e}^\ast) = L^2(\Omega)$, this condition reduces to the existence of a point of continuity of the form $L_T^\ast \adjtar$ for $F_{T}^\ast$.
In the cases covered here, we shall check that the chosen $F_{T}^\ast$ is continuous at~$0$. When strong duality holds, it is therefore equivalent to work with the dual problem, which is easier to handle especially when it has full domain, \ie its objective function is finite everywhere.

We note that an alternative to establish strong duality is to find $u \in E$ such that $G$ is continuous at $L_T u$ and $F(u)<+\infty$. This approach is bound to fail here since it would require finding a control achieving the target ball, \textit{i.e.}, assuming that controllability holds.

\paragraph{Non-trivial strong duality.} Furthermore, the primal value $\pi$ is attained if finite, \ie if this equality is not the trivial $+\infty = +\infty$ (the uncontrollable case). Thus, if $d$ is finite, $\pi$ is finite as well and attained: we may speak of optimal controls. 
%\begin{thrm}
%\label{general-approach}

%then system~\eqref{nonnegative-controllability} is approximately controllable with controls in $\U_2$, in time $T$. 
%\end{thrm}
This requirement that $d$ be finite is by far the subtlest one. It may be tackled by proving that the functional $\dualfct$ underlying the dual problem (written in infimum form $\textstyle \inf_{\adjtar \in L^2(\Omega)} \dualfct(\adjtar)$) has a minimum. In practice, we will always find this to be the case, as the dual problem is usually unconstrained (depending on the choice of $\cost$), unlike the primal problem.
Hence, both $\pi$ and $d$ will be attained and, from~Proposition~\ref{saddle-point-stuff}, any optimal dual variable $\adjtar^\star$ is such that any optimal control $u^\star$ satisfies
\begin{equation}\label{optimal-control-characterization}
    u^\star\in \partial \cost^\ast(L_T^\ast \adjtar^\star).
\end{equation}

%As a result, the above approximate controllability result in $\U_2$ has an obvious corollary in terms of extending controllability to $\U_1$. 
\begin{prpstn}
\label{prop-control-in-U}
Assume that, for any $y_0, \tar\in L^2(\Omega)$ such that $\tar\geq S_T y_0$ and any $\e>0$, 
\begin{itemize}
    \item there exists $\adjtar \in L^2(\Omega)$ such that $\cost^\ast$ is continuous at $L_T^\ast \adjtar$, 
   \item $d \neq +\infty$.
   \end{itemize}
   If for any dual optimal variable~$\adjtar^\star$, the  controls characterised by \eqref{optimal-control-characterization} are in $\U_+$, then the control system~\eqref{nonnegative-controllability} is nonnegatively approximately controllable with controls in $\U_+$ in time $T$.
\end{prpstn}
This shows how the choice of the cost $\cost$ impacts the existence and properties of optimal controls. More precisely, it must be pointed out that all the hypotheses of Proposition \ref{prop-control-in-U} are formulated with respect to the dual problem. Accordingly, from the next section onwards, our strategy will be to determine an adequate optimal control problem by designing its dual problem.

Finally, we emphasise that \eqref{optimal-control-characterization} is only a necessary condition for the optimality of $u^\star$. It becomes sufficient only when $\partial \cost^\ast(L_T^\ast \adjtar^\star)$ is reduced to a singleton, which will occur in our case.

\subsection{Convex analytic interpretation of the bathtub principle}
\label{subsec-bathtub-convex-analysis}
Starting from the set of on-off shape controls of amplitude $1$, 
\begin{equation}\label{1-shapess-L}\U_L:=\{\chi_\omega, \quad \omega\subset \Omega, \quad |\omega|\leq \mL\},\end{equation}
where $|\cdot|$ denotes the Lebesgue measure, we define the closure of its convex hull (which is also its weak-$\ast$ closure for the $L^\infty(\Omega)$-topology)
\begin{equation}\label{convex-constraint-set}
\overline{\U}_L:= \left\{ u \in L^2(\Omega), \; 0 \leq u \leq 1  \text{ and } \int_\Omega  u \leq \mL\right\}.\end{equation}

Given a fixed $v \in L^2(\Omega)$, we consider the (static) maximisation problem
\begin{equation}
\label{bathtub-relaxed}\sup_{u \in \overline{\U}_L} \int_\Omega u(x) v(x) \,dx.\end{equation}
This a relaxed version of the so-called bathtub principle, which gives the maximum value as well as a characterisation of maximisers~\cite{Lieb2001}. For the sake of readability, we introduce the necessary results for what follows, but refer to Appendix \ref{app-sec-bathtub} for a more detailed statement. 
For a given $v \in L^2(\Omega)$, we let
\begin{equation}\label{Phi}
    \Phi_v(r):= \left|\{v>r\}\right|.
\end{equation}
and its pseudo-inverse function
\begin{equation}
    \label{Phi-1}
    \Phi^{-1}_v(s):=\inf_{r\in \R}\left\{\Phi_v(r)\leq s\right\}=\inf_{r\in \R}\left\{|\{v>r\}|\leq s\right\}.
\end{equation}

Finally, we set \begin{equation}
\label{level-set-function}
h(v) :=  \max(0, \Phi^{-1}_v(\mL)). 
\end{equation}
\begin{rmrk}
The function $\Phi_v^{-1}$ is the Schwarz rearrangement of $v$, see~\cite{kawohl2006rearrangements}.
\end{rmrk}

\begin{lmm}[relaxed bathtub principle]
\label{relaxed_bathtub}
Let $v \in L^2(\Omega)$. The maximum in~\eqref{bathtub-relaxed} equals
\begin{equation*}
\int_0^{\min(\Phi_v(0), \mL)} \Phi_v^{-1}.
\end{equation*}
Furthermore, if all the level sets of the function $v$ have measure zero, the maximum equals $\textstyle \int_{\{v>h(v)\}} v$ and is uniquely attained by
\begin{equation*} u^\star := \chi_{\{v>h(v)\}}, \end{equation*}
\end{lmm}

We refer to Lemma \ref{lmm-app-bathtub-relaxed} for the comprehensive statement of the relaxed bathtub principle.
We may interpret the above results as a formula for the support function of $\overline{\U}_L$ in $L^2(\Omega)$:
\begin{equation}
\label{conjugate_formula} 
\sigma_{\overline{\U}_L}(v)=\sup_{u \in \overline{\U}_L} \left( \langle u,v \rangle_{L^2} - \delta_{\overline{\U}_L}(u)\right))=\sup_{u \in \overline{\U}_L} \int_\Omega u(x) v(x) \,dx  = \int_0^{\min(\Phi_v(0), \mL)} \Phi_v^{-1}.  \end{equation}

First, using the characterisation of the subdifferential given in Appendix~\ref{app-convex-analysis}, we arrive at the following characterisation for the solutions to the maximisation problem given in Lemma \ref{relaxed_bathtub}:
\begin{prpstn}\label{prop-bathtub-subdiff}
Let $v\in L^2(\Omega)$. The maximisers of the relaxed bathtub problem are given by the elements of $\partial \sigma_{\overline{\U}_L}(v)$.
\end{prpstn}
\begin{proof}
Using $(\delta_{\overline{\U}_L})^\ast= \sigma_{\overline{\U}_L}$ along with~\eqref{subdiff-and-convex-conjugate} given in Appendix \ref{app-subsec-conj}, we have, for $v\in L^2(\Omega)$,
\begin{equation*}
        \argmax_{u\in \overline{\U}_L} \langle u, v \rangle_{L^2} = \argmax_{u\in L^2} \langle u, v \rangle_{L^2} - \delta_{\overline{\U}_L}(u)
        =\left\{u \in L^2, \quad \langle u, v \rangle_{L^2} - \delta_{\overline{\U}_L}(u)=\sigma_{\overline{\U}_L}(v)\right\}=\partial\sigma_{\overline{\U}_L}(v).
\end{equation*} 
\end{proof}
\begin{rmrk}
Proposition \ref{prop-bathtub-subdiff} implies that for any maximiser $u$ of the relaxed bathtub problem,
\[v\in \partial \delta_{\overline{\U}_L} (u).\]
Proposition \ref{prop-indicator-subdiff} in Appendix \ref{subsec-app-indic} shows that this implies $u\in \partial\overline{\U}_L$. Propositions \ref{relaxed_bathtub} and \ref{prop-bathtub-subdiff} characterise exactly which elements of the boundary $\partial \overline{\U}_L$ are involved.

\end{rmrk}

\subsection{From the static bathtub principle to the dual problem and its corresponding cost}
\label{subsec-static-to-dynamic}

Following Section \ref{subsec-abstract-control} and
recalling Proposition \ref{prop-control-in-U} and \eqref{optimal-control-characterization}, we are looking for a cost function $\cost$ such that the corresponding optimal controls are on-off shape controls, and we have established that it suffices to find a conjugate functional $\cost^\ast$ satisfying two key properties. First, if there exists $\adjtar\in L^2(\Omega)$ such that $\cost^\ast$ is continuous at $L_T^\ast \adjtar$, and if we can provide the existence of a minimiser $\adjtar^\star$ of $\dualfct$, then $\pi$ is attained and there exists at least one optimal control. Second, any optimal control~$u^\star$ should satisfy~\eqref{optimal-control-characterization}
so~$\cost^\ast$ should be chosen so that the subdifferential $\partial \cost^\ast(L_T^\ast \adjtar^\star)$ contains only characteristic functions. Given Proposition \ref{prop-indicator-subdiff} and Section \ref{subsec-bathtub-convex-analysis}, elements of 
\[\partial \sigma_{\overline{\U}_L}(v), \quad v\in L^2(\Omega)\]
are bang-bang, in the sense that they are characteristic functions, under some mild conditions that $v$ must satisfy.

To go from the static optimisation problem to the adequate dual problem, we add a time dependency. Moreover, to ensure coercivity of the dual problem, we add a quadratic exponent. All in all, we choose the following conjugate:
%\begin{equation}
%\label{conj_cost_1}
%F_1^\ast(p):=\int_0^T \frac12 \left(\sigma_{\overline{\U}_L} (p(t))\right)^2dt = \frac12 \int_0^T \left(  \int_0^{\min (\Phi_{p(t)}(0), \mL)} \Phi_{p(t)}^{-1}\right)^2 dt, \quad \forall p \in E,
%\end{equation}
\begin{equation}
\label{conj_cost_2}
\cost^\ast(p) :=\frac12 \left(\int_0^T \sigma_{\overline{\U}_L}(p(t)) \,dt\right)^2 =  \frac12 \left(\int_0^T  \int_0^{\min (\Phi_{p(t)}(0), \mL)} \Phi_{p(t)}^{-1}(s) \,ds \,dt\right)^2 , \quad \forall p \in E.\end{equation}

Since the approximate controllability problem corresponds to $G_{T,\e}^\ast:=\sigma_{\overline{B}(\tar - S_T y_0, \e)}$, this defines a dual problem of the form \eqref{dual-control-problem}.
As pointed out in Section \ref{subsec-abstract-control}, we are now dealing with an unconstrained optimisation problem (\ie the domain of the functions involved is the whole space $L^2(\Omega$)). 

We can now derive the corresponding constrained optimisation problem, by computing the actual cost~$\cost$ associated to the choice~\eqref{conj_cost_2} for $\cost^\ast$ . We find, as announced by \eqref{intro-cost-function} in the introduction:
\begin{lmm}
\label{lem-compute-costs}
The function $\cost^\ast$ defined by \eqref{conj_cost_2} satisfies $\cost^\star \in \Gamma_0(E)$. Defining
\begin{equation*}
%\label{control-amplitude}
\mathcal{M}(u):=\max\left(\|u\|_{L^\infty},  \frac{\|u\|_{L^1}}{\mL}\right),\quad \forall u \in L^2(\Omega),\end{equation*}
its Fenchel conjugate $(\cost^\ast)^\ast = \cost$ is given for $u \in E$ by 
\[\cost(u) =\frac12 \bigg(\sup_{t \in [0,T]}  \mathcal{M}^2(u(t, \cdot))\bigg) + \delta_{\{u \geq 0\}}(u) = \frac12 \bigg(\sup_{t \in [0,T]}  \max\Big(\|u(t, \cdot)\|_{L^\infty}, \frac{\|u(t, \cdot)\|_{L^1}}{\mL}\Big)^2 \bigg)+ \delta_{\{u \geq 0\}}(u).\]
%\begin{equation}
%    \label{cost_1}
%C_1(u) := \frac12 \int_0^T  \max\left(\|u(t)\|_\infty, \frac{\|u(t)\|_1}{\mL}\right)^2 dt,
%\end{equation}
\end{lmm}

\begin{proof}
Lemma~\ref{lem-integral} in Appendix \ref{app-subsec-lemmas} shows that $\cost^\ast \in \Gamma_0(E)$. 
We proceed by computing $(\cost^\ast)^\ast$. We have $\cost^\ast={\textstyle\frac12} H^2$, with $H(p):=\textstyle \int_0^T \sigma_{\overline{\U}_L}(p(t, \cdot)) \, dt$. Since $\sigma_{\overline{\U}_L} \in \Gamma_0(L^2(\Omega))$, the definition of the support function together with Lemma~\ref{lem-integral} in Appendix  show that $H \in \Gamma_0(E)$ with
\[H^\ast(u) = \int_0^T \sigma_{\overline{\U}_L}^\ast(u(t, \cdot))\,dt =  \int_0^T \delta_{\overline{\U}_L}(u(t, \cdot))\,dt.\]
Furthermore, we find the conjugate of $\textstyle \frac{1}{2} H^2$ by using~\eqref{square-conjugate} in Appendix~\ref{app-subsec-conj}, which leads to
\[\left(\frac12 H^2\right)^\ast (u) = \min_{\alpha>0} \left(\frac12 \alpha^2 + \alpha H^\ast\left(\frac{u}{\alpha}\right)\right),\]
where we used that $\mathrm{dom}(H)= E$.
Clearly, \[H^\ast\left(\frac{u}{\alpha}\right) = 0\quad \text{if} \quad u \geq 0 \; \text{ and } \; \sup_{t \in [0,T]}  \mathcal{M}(u(t, \cdot)) \leq \alpha,\]
and is $+\infty$ otherwise.

We end up with
\[\left(\frac12 H^2\right)^\ast (u) = \min_{\alpha>0} \left(\frac12 \alpha^2 + \delta_{\left\{\sup_{t \in [0,T]}\mathcal{M}(u(t, \cdot)) \leq \alpha\right\}}(u) \right) + \delta_{\{u \geq 0\}}(u) = \cost(u).\]
The lemma is proved.
\end{proof}

Note that $\cost^\ast$ is
(positively)-homogeneous of degree $2$. Indeed, $v \mapsto \sigma_{\overline{\U}_L}$ is positively-homogeneous of degree $1$, \ie $\sigma_{\overline{\U}_L}(\lambda v) = \lambda \sigma_{\overline{\U}_L}(v)$ for all $\lambda>0$, $ v\in L^2(\Omega)$.

We end this subsection by establishing a crucial property satisfied by $\cost^\ast$. It will play a crucial role in proving that the dual functional is coercive, akin to that of the unique continuation property in the Lions strategy described in Section \ref{subsec-results}.
\begin{lmm}
\label{continuation}
For all $\adjtar  \in L^2(\Omega)$, if $\cost^\ast(L_T^\ast \adjtar) = 0$, then $\adjtar \leq 0$.
\end{lmm}
\begin{proof}
It is easily seen that $\sigma_{\overline{\U}_L}\geq 0$, and, for $v\in L^2(\Omega)$, $\sigma_{\overline{\U}_L}(v)> 0$ as soon as $v>0$ on a set of positive measure, by taking the scalar product of $v$ against a well chosen element of $\overline{\U}_L$. Consequently, if $\sigma_{\overline{\U}_L}(v)=0$, then $v\leq0$.

Now recall that
\[F_T^\ast(L_T^\ast \adjtar)=\frac12 \left(\int_0^T \sigma_{\overline{\U}_L}(L_T^\ast \adjtar(t, \cdot)) \,dt\right)^2 .\]
If $F_T^\ast(L_T^\ast \adjtar)=0$, then, as $\sigma_{\overline{\U}_L}\geq 0$, for almost every $t\in (0,T)$, $\sigma_{\overline{\U}_L}(L_T^\ast \adjtar)=0$. Thus, $L_T^\ast p_f(t, \cdot)\leq 0$. In particular, since $L_T^\ast \adjtar \in C([0,T]; L^2(\Omega)) $, this implies that $p_f\leq 0.$
\end{proof}

\section{Approximate controllability results}

\label{sec-controllability}

In this section, we state and prove our main result on approximate controllability. The full statement for our Theorem~\ref{informal-shape} is given with more details below, for general linear operators, satisfying the properties given in Section \ref{subsec-general-results}.

We are considering the following optimal control problem:
\begin{equation}
    \label{ocp-primal}
\pi= \inf_{u \in E} \cost(u) + G_{T,\e}(L_T u)=  \inf_{u \in E} \left\{ \frac12 \sup_{t\in[0,T]} \max \left(\|u(t)\|_{L^\infty} , \frac{\|u(t)\|_{L^1}}{\mL}\right)^2 + \delta_{\overline{B}(\tar - S_T y_0, \e)}(L_T u)\right\},
\end{equation}
whose dual problem is
\begin{equation}
    \label{ocp-dual}
    d=-\inf_{\adjtar\in L^2(\Omega)} \dualfct(\adjtar)=-\inf_{\adjtar \in L^2(\Omega)} \left\{\frac12 \left(\int_0^T \sigma_{\overline{\U}_L}(L_T^\ast \adjtar (t))\, dt\right)^2 - \langle \tar - S_T y_0, \adjtar\rangle_{L^2}+ \e \|\adjtar\|_{L^2}\right\}.
\end{equation}
\begin{thrm}\label{thm-controllability}
Assume that $A^\ast$ satisfies the ~\eqref{main_property} property  and that $\partial_t - A^\ast$ is analytic-hypoelliptic. Then for the cost function $\cost$ defined by~\eqref{intro-cost-function},
\begin{itemize}
    \item the strong duality $\pi = d$ holds,
    \item the dual problem \eqref{ocp-dual} is attained at a unique minimiser $\adjtar^\star\in L^2(\Omega)$,
    \item there exists a unique optimal control~$u^\star \in E$ for the primal problem \eqref{ocp-primal}.
\end{itemize}

Furthermore,
%, if the target is not reached by the trivial control $u\equiv 0$, \ie if
%\begin{equation}
%    \label{non-trivial-target}
%   \tar \notin \overline{B}(S_T y_0, \e),
%\end{equation} 
%then
the optimal control is given by 
%\begin{equation}
%\label{shape-1}
%u^\star(t) = M(t)  \chi_{\{p^\star(t) > r(p^\star(t))\}}, \quad M(t) = \int_{\{p^\star(t) > r(p^\star(t))\}} p^\star(t)  \end{equation}
\begin{equation} 
\label{shape-2}
u^\star(t,\cdot) = \overline{M} \, \chi_{\{p^\star(t,\cdot) > h(p^\star(t,\cdot))\}},\quad \overline{M} = \int_0^T \int_{\{p^\star(t,\cdot) > h(p^\star(t,\cdot))\}} p^\star(t,x) \,dt \, dx, \end{equation}
where $h$ is defined by~\eqref{level-set-function}, and where $p^\star=L_T^\star \adjtar^\star$ is the solution of the adjoint equation~\eqref{backward} satisfying $p^\star(T)=\adjtar^\star$.
\end{thrm}

\begin{rmrk}
In fact, if $\e>0$ is such that  $\tar \in \overline{B}(S_T y_0, \e)$, we prove that $\adjtar^\star = 0$ and the formula above returns $u^\star = 0$, which obviously does steer the system to the target ball. 
\end{rmrk}
 
\begin{rmrk}
As mentioned in the introduction, Theorem \ref{thm-controllability} holds for uniformly elliptic operators of the form \eqref{def_elliptic} with analytic coefficients, and in particular the classical heat equation with Dirichlet boundary conditions, on a bounded, open, connected domain with $C^2$ boundary. 
\end{rmrk}
%\begin{rmrk}
%Controllability through controls of the form (i) might at first glance look like a weaker statement since their amplitude is time-dependent. However, at all times $t \in (0,T)$, they are uniquely determined by the adjoint trajectory at the same time $t$. On the other hand, controls of the form (ii) have time-independent amplitude, but computing the amplitude requires the full adjoint trajectory $p(t), t \in (0,T)$. Hence, we believe both results are worth stating. 
%\end{rmrk}

Throughout this section, we assume the hypotheses sufficient for Theorem~\ref{thm-controllability}, \ie that $A^\ast$ satisfies the~\eqref{main_property} property and that $\partial_t - A^\ast$  is analytic-hypoelliptic. The proof is then scattered into the section as follows:
\begin{itemize}
    \item First, we establish that strong duality holds.
    \item Second, we prove that the corresponding dual functional is coercive: hence, the dual functional attains its minimum (the dual problem attains its maximum).
    \item Third
    we prove~\eqref{shape-2}.
    \item Finally, we investigate the uniqueness of optimal variables.
\end{itemize}

\begin{rmrk}
\label{rmk-general-existence}
As the proofs show, the first two steps and the uniqueness of dual optimal variables are valid for any operator $A$. In particular, they do not require that $A^\ast$ satisfy the~\eqref{main_property} property and that $\partial_t - A^\ast$  be analytic-hypoelliptic. Hence, strong duality and existence of optimal controls does not require any specific assumption the semigroup must satisfy. This remark will be of importance in the next subsection where we manipulate optimal controls without making these two hypotheses. 
\end{rmrk}

\subsection{Strong duality}

\begin{lmm}\label{lem-contact-condition}
$\cost^\ast$ is continuous at $0 = L_T^\ast 0$.
\end{lmm}
\begin{proof}
By the Cauchy-Schwarz inequality, 
\[\forall u \in \overline{\U}_L, \quad \langle u,v \rangle_{L^2} \leq \|u\|_{L^2} \|v\|_{L^2} \leq  |\Omega|^{1/2} \|v\|_{L^2}, \]
which leads to 
 \[\sigma_{\overline{\U}_L}(v) = \int_0^{\min (\Phi_v(0), \mL)} \Phi_v^{-1}\leq  |\Omega|^{1/2} \|v\|_{L^2}.\]
As a result, we may bound  with the Cauchy-Schwarz inequality again
\[0 \leq \cost^\ast(p) \leq \frac12 |\Omega| \left(\int_0^T \|p(t, \cdot)\|_{L^2}\, dt\right)^2 \leq \frac12 T |\Omega| \,\|p\|_E^2,\] 
hence the continuity of $\cost^\ast$ at $0 = L_T^\ast 0$.
\end{proof}

The above lemma shows that the first condition of Proposition~\ref{prop-control-in-U} is satisfied, \ie strong duality holds.

\subsection{Coercivity of $\dualfct$, nonnegative approximate controllability}

\begin{prpstn}
\label{prop-coercive}
The functional $\dualfct$ defined by
\begin{equation}
\label{dualfct}
\dualfct(\adjtar) =\left(\frac12 \int_0^T   \sigma_{\overline{\U_L}}(L_T^\ast \adjtar) dt\right)^2- \langle \tar - S_T y_0, \adjtar\rangle_{L^2} + \varepsilon \|\adjtar\|_{L^2}.
\end{equation}
is coercive on $L^2(\Omega)$, \ie
\[   \dualfct(\adjtar) \xrightarrow[\|\adjtar\|_{L^2}\to \infty]{} \infty,\]
and thus attains its minimum.
\end{prpstn}
\begin{proof}
Since we know that $\dualfct$ is convex, proper, strongly lsc, if $\dualfct$ is coercive then $\textstyle \inf_{\adjtar \in L^2(\Omega)} \dualfct(\adjtar) \neq -\infty$, and that it is actually attained.

We will actually prove a stronger condition than coercivity, namely
\begin{equation*}
    %\label{strong-coercivity}
    \liminf_{\|\adjtar\|_{L^2}\to\infty} \frac{\dualfct(\adjtar)}{\|\adjtar\|_{L^2}}>0. 
\end{equation*}
Our method of proof follows that of~\cite{micu_introduction_2004, Kunisch-Wang-2013}. Take a sequence $\|\adjtar^n\|_{L^2}\to \infty$. We denote $\textstyle q_f^n := \frac{\adjtar^n}{\|\adjtar^n\|_{L^2}}$.
%and $q^n \in E$ the corresponding solution of the adjoint equation~\eqref{backward}, \ie such that $q^n(T) =q_f^n$, which by linearity is $\textstyle \frac{p^n}{\|\adjtar^n\|_{L^2}}$, where $p^n=L_T^\ast \adjtar^n$ is the solution of \eqref{backward} such that $p^n(T)=\adjtar^n$.
By positive homogeneity of $\cost^\ast$ (of degree $2$), we have
\begin{equation*}
    \frac{\dualfct(p^n_f)}{\|p^n_f\|_{L^2}}= \|p^n_f\|_{L^2} \cost^\ast(L_T^\ast q_f^n)-\left\langle \tar - S_T y_0, q_f^n\right\rangle_{L^2} + \varepsilon
\end{equation*}
and hence if $\liminf_{n\to \infty} \cost^\ast(L_T^\ast q_f^n)>0$, then \begin{equation*} \liminf_{n\to \infty} \frac{\dualfct(p^n_f)}{\|p^n_f\|{L^2}}=+\infty.\end{equation*}
Let us now treat the remaining case where $\liminf_{n\to \infty} \cost^\ast(L_T^\ast q_f^n)= 0$. Since $\|q_f^n\|_{L^2} = 1$, upon extraction of a subsequence, we have $q_f^n \rightharpoonup q_f$ weakly in $L^2(\Omega)$ for some $q_f \in L^2(\Omega)$.
Since $L_T^\ast \in L(L^2(\Omega),E)$, we have $L_T^\ast q_f^n \rightharpoonup L_T^\ast q_f$ weakly in $E$.

Now, since $\cost^\ast$ is convex and strongly lsc on $E$, it is (sequentially) weakly lsc and taking the limit we obtain $\cost^\ast(L_T^\ast q_f) = 0$.
By Lemma~\ref{continuation}, we infer that $q_f \leq 0$.

Then, recalling that the target satisfies $\tar - S_T y_0 \geq 0 \Leftrightarrow \tar \geq S_T y_0$, we end up with 
\begin{equation*}
   \liminf_{n\to \infty} \frac{\dualfct(p^n_f)}{\|p^n_f\|_{L^2}}\geq  -\langle \tar - S_T y_0, q_f \rangle_{L^2}+ \varepsilon  \geq \varepsilon >0,
\end{equation*}
which concludes the proof.
 \end{proof}

\subsection{Characterisation of the minimisers}
In this subsection and the subsequent one, it will be convenient to discuss depending on the assumption
\begin{equation}
\label{non-trivial-target}
       \tar \notin \overline{B}(S_T y_0, \e),
\end{equation}
in which case the target is not reached with the trivial control $u=0$.
Note that, if~\eqref{non-trivial-target} is not satisfied, the control $u=0$ steers $y_0$ to the target, and is indeed a control in~$\U_{\textrm{shape}}^L$. 

We first remark the following fact:
\begin{lmm}\label{lem-non-zero-dual-optimal} 
Assumption~\eqref{non-trivial-target} holds if and only if any minimiser $\adjtar^\star$  of~\eqref{ocp-dual} satisfies $\adjtar^\star \neq 0$.
\end{lmm}
\begin{proof}
Suppose that $\adjtar^\star=0$ minimises~\eqref{ocp-dual}. Then, $d=0$ and by strong duality, $\pi=0$. By Proposition~\ref{prop-coercive}, this value is attained: there exists some optimal control $u^\star$ such that
\[\cost(u^\star)+G_{T,\e}(L_T u^\star)=0.\]
This implies that $\cost(u^\star) = G_{T,\e}(L_T u^\star) = 0$.
On the one hand, this leads to $u^\star(t, \cdot) = 0$ for a.e $t\in (0,T)$, \ie $u^\star=0$,
and on the other hand $G_{T,\e}(L_T u^\star) = G_{T,\e}(0) =0$, which is equivalent to $0 \in \overline{B}(\tar - S_T y_0, \e) \iff y_f \in \overline{B}(S_T y_0, \e)$. This contradicts Assumption \eqref{non-trivial-target}.

Conversely, if Assumption~\eqref{non-trivial-target} does not hold, then $u=0$ drives $y_0$ to the target ball, hence $\pi=d=0$. Since $\dualfct(0)=0$, $\adjtar^\star = 0$ minimises the dual problem~\eqref{ocp-dual}.
\end{proof}

\begin{prpstn}

\label{prop-second-cost}
%Under Assumption~\eqref{non-trivial-target}, 
Any optimal control for~\eqref{ocp-primal} is of the form~\eqref{shape-2}, where $p^\star$ denotes the solution of the adjoint equation~\eqref{backward} such that $p^\star(T)=\adjtar^\star$, where $\adjtar^\star$ is any dual optimal variable.
\end{prpstn}
\begin{proof}

Let $u^\star$ be an optimal control. 
Thanks to Proposition \ref{prop-coercive}, we know that $\dualfct$ defined by~\eqref{ocp-dual} attains its minimum.
Thanks to Lemma \ref{lem-contact-condition}, we can apply the first identity of \eqref{dual-to-primal-variable} in Proposition \ref{saddle-point-stuff} (see Appendix~\ref{app-subsec-fenchel}) to obtain~$u^\star \in \partial \cost^\ast(L_T^\ast \adjtar^\star)$, where $\adjtar^\star$ is any minimiser of $\dualfct$, \ie an optimal dual variable.  

We denote $p^\star$ the solution of the adjoint equation~\eqref{backward} such that $p^\star(T)=\adjtar^\star$. From Lemma~\ref{lem-non-zero-dual-optimal},  
Using again the notation $H(p):=\textstyle \int_0^T\sigma_{\overline{\U}_L}(p(t)) \,  dt$, so that $\cost^\ast=\textstyle \frac12 H^2$,
we have $H(p^\star)\geq 0$ and $\operatorname{dom}(H)=L^2(\Omega)$.

Then, applying the generalised chain rule (see \cite[Theorem 2.3.9, point (ii)]{clarke1990optimization}) with the functions $x\mapsto \textstyle \frac{1}{2}x^2$ and $H$, we compute the subdifferential of the convex functional~$\cost^\ast$: 
$u^\star   \in  H(p^\star) \,\partial   H(p^\star).
$
Applying Lemma~\ref{lem-integral} to $H$, we find $u^\star(t, \cdot) \in \overline{M} \partial \sigma_{\overline{\U}_L}(p^\star(t,\cdot))$ for a.e. $t\in (0,T)$, with $\overline{M} := H(p^\star)$.

Let us first assume that Assumption~\eqref{non-trivial-target} holds. From Lemma~\eqref{lem-non-zero-dual-optimal}, we have $\adjtar^\star \neq 0$. We now let $t \in (0,T)$ be fixed and let us justify that all level sets of $p^\star(t,\cdot)$ are of measure zero, \ie
\[|\{p^\star(t, \cdot)=\lambda\}|=0, \quad \forall \lambda\in \R,\]
Indeed, since the operator $\partial_t -A^\ast$ is analytic-hypoelliptic, we know that $p^\star(t, \cdot)$ is analytic on $\Omega$. Hence, its level sets are of measure zero unless $p^\star(t, \cdot) = S_{T-t}^\ast \, \adjtar^\star$ is constant. Using the~\eqref{main_property} property, this leads to $\adjtar^\star=0$, contradicting~\eqref{non-trivial-target}.
%Dirichlet boundary conditions then impose $p^\star(t) = 0$. By the parabolic maximum principle, this leads to $p^\star=0$, contradicting~\eqref{non-trivial-target}.

Applying Propositions \ref{relaxed_bathtub} and \ref{prop-bathtub-subdiff}, and recalling that $\partial \sigma_{\overline{\U}_L}(p^\star(t, \cdot))=\{\chi_{\{p^\star(t,\cdot)>h(p^\star(t, \cdot))\}}\}$,
we obtain the result.

Now assume that Assumption~\eqref{non-trivial-target} does not hold. Then $\adjtar^\star = 0$ is optimal and, using the above notations for this specific dual optimal variable, we have $p^\star = 0$, $\overline{M} = 0$ hence any optimal control satisfies $u^\star = 0$, which is of the form~\eqref{shape-2}.
\end{proof}

\begin{rmrk}
As evidenced by the proof, a weaker (but less workable) property than analytic-hypoellipticity is sufficient to infer that optimal controls are on-off shape controls. Indeed, it suffices to require either one of the following conditions (in decreasing order of strength):
\begin{enumerate}[label=(\roman*)]
    \item All solutions $t \mapsto p(t)$ of the adjoint equation such that $p(T) \neq 0$ have zero-measure level sets.
    \item For all solutions $t \mapsto p(t)$ of the adjoint equation such that $p(T) \neq 0$, the level sets $\{p(t, \cdot)=h(p(t, \cdot))\}$ (see~\eqref{level-set-function} for the definition of $h(p)$) have measure $0$.
    \item For all solutions $t \mapsto p(t)$ of the adjoint equation such that $p(T) \neq 0$,
\[\left\{\begin{aligned}
&|\{p(t, \cdot)=h(p(t, \cdot))\}|=\mL - |\{p(t, \cdot)>h(p(t, \cdot))\}| , \quad &\text{if} \  h(p(t,\cdot))\neq 0 \\
&|\{p(t, \cdot)=h(p(t, \cdot))\}|=0, \quad &\text{if} \  h(p(t, \cdot))=0
\end{aligned}
 \quad \fae t\in [0,T].\right.\]
\end{enumerate}  
Note that requirement $(iii)$ is minimal (see Lemma \ref{lmm-app-bathtub-relaxed} and Remark \ref{app-rmrk-bathtub-uniqueness}).

Finally, an even weaker requirement would be to restrict any of the above $(i)$, $(ii)$ or $(iii)$ to a single solution $t\mapsto p^\star(t)$ of the adjoint equation, namely that with $p^\star(T) =\adjtar^\star$ where  $\adjtar^\star$ is the unique dual optimal variable (see below for the uniqueness of optimal variables).
\end{rmrk}

\subsection{Uniqueness}
\label{subsec-uniqueness}
Our first uniqueness statement below (\ie that of the dual optimal variable) is a consequence of Fenchel-Rockafellar duality, and the fact that we work with a Hilbert space, rather than specific properties of the evolution equation under consideration.
\begin{rmrk}
\label{rmrk_bndr}
Still applying Proposition \ref{saddle-point-stuff}, we get 
\begin{equation*}
    L_T u^\star \in \partial G_{T,\e}^\ast(-\adjtar^\star)=\partial \sigma_{\overline{B}(\tar - S_T y_0, \e)}(-\adjtar^\star).
\end{equation*}
Using the Legendre-Fenchel identity \eqref{flip-subdiff}, we get $-\adjtar^\star \in \partial \delta_{\overline{B}(\tar - S_T y_0, \e)} (L_T u^\star)$.
Thanks to Proposition~\ref{prop-indicator-subdiff}, this means that $L_T u^\star$ lies at the boundary of the closed ball $\overline{B}(\tar - S_T y_0, \e)$.
\end{rmrk}

\begin{prpstn}\label{prop-uniqueness-p-u}
Under the assumptions of Theorem \ref{thm-controllability}, the primal-dual optimal pairs $(u^\star, \adjtar^\star)$ are unique.  
\end{prpstn}
\begin{proof}

\textbf{Uniqueness of the dual optimal variable.}
First note that if Assumption \eqref{non-trivial-target} does not hold, then $0$ is the unique optimal control, \ie 
\begin{equation}\label{unique-target-trivial}\{L_T u^\star,\, u^\star\text{ is optimal}\}=\{0\}.\end{equation}

On the other hand, if Assumption \eqref{non-trivial-target} holds, according to Remark~\ref{rmrk_bndr}, and since the set of minimisers of a convex function is convex, the set $\{L_T u^\star, \, u^\star \text{ is optimal}\}$ is a convex subset of the sphere $S(\tar - S_T y_0, \e)$. The closed ball being strictly convex since we are working in the Hilbert space $L^2(\Omega)$, there exists some $y^\star \in \overline{B}(\tar - S_T y_0, \e)$ with $\|y^\star - (\tar - S_T y_0)\|_{L^2} = \e$ such that \begin{equation}\label{unique-target}
\{L_T u^\star, \, u^\star \text{ is optimal}\} = \{y^\star\}.\end{equation}
Thus, in any case, the set of targets reached by optimal controls is always reduced to a single point.

Now, let $\adjtar^\star$ be a dual optimal variable, and $u^\star$ an optimal control. Then, as strong duality holds, Proposition \ref{saddle-point-stuff-weak} implies that the pair $(u^\star, \adjtar^\star)$ satisfies the two optimality conditions from \eqref{unique-target}. We then have
    \begin{equation}
    \label{p-T-charac-unique}
    \adjtar^\star \in -\partial G_{T,\e}(L_T u^\star) = - \partial \delta_{\overline{B}(\tar - S_T y_0, \e)}(L_T u^\star).\end{equation}
    
If Assumption \eqref{non-trivial-target} does not hold, then \eqref{p-T-charac-unique} and \eqref{unique-target-trivial} imply $\adjtar^\star \in -\partial\delta_{\overline{B}(\tar - S_T y_0, \e)}(0)$.
If $\|\tar - S_T y_0\|_{L^2}<\e$, then $0\in B(\tar - S_T y_0, \e)$ and 
\begin{equation}\label{non-H-case-1}\adjtar^\star \in -\partial\delta_{\overline{B}(\tar - S_T y_0, \e)}(0)=\{0\}.\end{equation}
Otherwise, $0\in \partial B (\tar - S_T y_0, \e)$ and \eqref{subdiff-indic} yield
\[\adjtar^\star \in\left\{\lambda \frac{\tar - S_T y_0}{\e}, \lambda \geq 0\right\} = \{\lambda (\tar - S_T y_0), \lambda \geq 0\}.\]
Restricting the function $\dualfct$ defining the dual problem~\eqref{ocp-dual} to the above half-line, using the homogeneities of each of its terms, and the fact that $\|\tar - S_T y_0\|_{L^2}=\e$, we get
\begin{equation}\label{non-H-case-2}\gamma_0(\lambda):=\dualfct(\lambda (\tar - S_T y_0))=a_0 \lambda^2, \quad \lambda\geq 0.\end{equation}
It is clear that $0$ is the unique minimiser of $\gamma_0$.
From \eqref{non-H-case-1} and \eqref{non-H-case-2}, $0$ is the unique dual optimal variable if \eqref{non-trivial-target} does not hold.

If Assumption \eqref{non-trivial-target} holds, then \eqref{p-T-charac-unique} and \eqref{unique-target} imply
\[\adjtar^\star\in - \partial \delta_{\overline{B}(\tar - S_T y_0, \e)}(y^\star) = -\partial \delta_{\overline{B}(0,1)}\left(\frac{y^\star-(\tar - S_T y_0)}{\e}\right).\]
Since $y^\star$ lies at the boundary of $\overline{B}(\tar - S_T y_0, \e)$, formula~\eqref{subdiff-indic} yields
\[\adjtar^\star \in\left\{\lambda \left(\frac{\tar - S_T y_0-y^\star}{\e}\right), \lambda \geq 0\right\} = \{\lambda \left(\tar - S_T y_0 - y^\star\right), \lambda \geq 0\}.\]
Restricting $\dualfct$ to the above half-line as previously, we find
\[\gamma(\lambda):=\dualfct(\lambda (\tar - S_T y_0-y^\star)) = a \lambda ^2 + b \lambda, \quad \lambda \geq 0,\]
where, using $\|\tar - S_T y_0 - y^\star\|_{L^2} = \e$ and the homogeneities involved $
a= \cost^\ast(L_T^\ast (\tar - S_T y_0-y^\star))$ and $b=-\langle \tar - S_T y_0, \tar - S_T y_0-y^\star  \rangle_{L^2} + \e^2$.
By coercivity, $a >0$, and given Lemma \ref{lem-non-zero-dual-optimal}, we have $b<0$. 

Thus, $\gamma$ has a unique minimiser $\lambda^\star:=-b/2a>0$. Hence, $\adjtar^\star=\lambda^\star(\tar - S_T y_0-y^\star)$,
and the dual optimal variable is unique.

\paragraph{Uniqueness of the  optimal control.}

If Assumption \eqref{non-trivial-target} does not hold, then~$0$ is the unique optimal control.

Now, suppose that Assumption~\eqref{non-trivial-target} holds. We know from the proof of Proposition \ref{prop-second-cost} that a given dual optimal variable uniquely determines one optimal control. Moreover, as we have proved that strong duality holds, we can apply Proposition \ref{saddle-point-stuff-weak}: for any pair of primal and dual optimal variables, the relations \eqref{saddle-x-identities} are satisfied. That is, any optimal control $u^\star$ is uniquely determined by the unique dual optimal variable $\adjtar^\star$ through the identity $u^\star \in \partial \cost^\ast(L_T^\ast \adjtar^\star)$.
\end{proof}

\section{Obstructions to controllability}
\label{sec-obstructions}
%\subsection{Acting on the whole domain is necessary}
We here prove Theorem~\ref{informal-obstruction}, through the more general result below in the case of second-order uniformly elliptic operators of the form~\eqref{def_elliptic}. We use the notation $A \subset \subset B$ to mean that there exists a compact set $K$ such that $A \subset K \subset B$.
\begin{thrm}
\label{thm-obstruction}
Let $\mathcal  U_+ \subset L^2(\Omega)$ be a constraint set of nonnegative controls. Assume
that there exists a ball $B(x, r) \subset  \Omega$ such that \[\forall u \in \U_+, \quad \mathrm{supp}(u) \cap B(x, r) = \emptyset.\] Let $A$ be a second-order uniformly elliptic operator of the form~\eqref{def_elliptic}.
Let $y_0 = 0$ and $\tar \in L^2(\Omega)$ be any target such that $\tar \geq S_T y_0 = 0$, $\tar \neq 0$ and $\mathrm{supp}(y_f) \subset B(x, r)$. 
Then there exist $T^\star>0$ and $\e>0$ such that for any time $T\leq T^\star$, no control with values in $\U_+$ can steer $0$ to $\overline{B}(\tar, \e)$.
\end{thrm}
The proof relies on the following lemma, inspired by~\cite{Pighin2018}.
\begin{lmm}
\label{nice-adjoint}
Let $B(x, r) \subset \subset \Omega$. Under the assumptions of Theorem~\ref{thm-obstruction}, for any $K \subset B(x, r)$ compact, there exists $\adjtar \in L^2(\Omega)$ and $T^\star>0$ such that
\begin{itemize}
    \item[(i)] $\adjtar <0$ on $K$,
   % \item[(ii)] $\exists \, T^\star>0$ such that for all $t \in (0,T^\star)$, $p(t, \cdot) \geq 0$ on $\Omega \setminus B(x, r)$, where $p$ solves the adjoint equation~\eqref{backward} with $p(T)=\adjtar$.
    \item[(ii)] for all $T \leq T^\star$, the solution of~\eqref{backward} with $p(T)=\adjtar$ satisfies $p(t, \cdot) \geq 0$ on $\Omega \setminus B(x, r)$ for all $0 \leq t \leq T$.
\end{itemize}
\end{lmm}

\begin{proof}
Let us build $\adjtar$ such that for all $1<r<+\infty$, $\adjtar \in W^{2,r}(\Omega) \cap W^{1,r}_0(\Omega)$, with $\adjtar<0$ on $K$, $\adjtar >0$ on $\Omega \setminus B(x, r)$, $\adjtar= 0$ on $\partial \Omega$, and $\partial_\nu \adjtar <0$ on $\partial \Omega$. 

To that end, we denote $\varphi_1$ the first eigenfunction of the Dirichlet Laplacian on $\Omega$, which satisfies $\varphi_1>0$ on $\Omega$ and $\partial_\nu \varphi_1<0$ on $\partial \Omega$ and since $\Omega$ is of class $C^2$, $\varphi_1 \in W^{2,r}(\Omega) \cap W^{1,r}_0(\Omega)$ for all $1<r<+\infty$~\cite{Brezis2011}[Theorem 9.32]. We then set $\adjtar = \xi \varphi_1$ where $\xi \in C^\infty(\Omega)$ is chosen to satisfy $\xi = 1$ on $\Omega \setminus B(x,r)$ and $\xi = -1$ on $K$. The function $\adjtar$ satisfies all the required properties (note that $\adjtar = \varphi_1$ locally around $\partial \Omega$ since $B(x,r)\subset \subset \Omega$, hence $\partial_\nu \adjtar =\partial_\nu \varphi_1 < 0$ on $\partial \Omega$).

Let $q$ solve the (forward) adjoint equation~\eqref{backward} for $t\geq 0$, with $q(0) =\adjtar$.
Then, by parabolic regularity, we both have $q \in C([0,+\infty)\times \overline \Omega)$ and $\partial_\nu q \in  C([0,+\infty)\times \partial \Omega)$~\cite{Pighin2018}[Theorem 8.1]. As a result, by continuity there exists $T^\star$ such that $\partial_\nu q<0$ over $[0,T^\star]\times \partial \Omega$, hence there exists some compact set~$K_1$ containing $B(x, r)$ such that $q \geq 0$ on $[0,T^\star]\times (\Omega \setminus K_1)$. 
Then, upon reducing $T^\star$ if necessary and by continuity again, we have $q \geq 0$ over $[0,T^\star] \times (K_1 \setminus B(x, r))$.

To conclude the proof, we fix any $T\leq T^\star$ and let $p$ be the solution of~\eqref{backward} on $[0,T]$ with $p(T)=\adjtar$. Then for all $0 \leq t \leq T$, $p(t) = q(T-t)$, hence $p(t, \cdot) \geq 0$ on $\Omega \setminus B(x, r)$ for all $0 \leq t \leq T$. 
\end{proof}

\begin{proof}[Proof of Theorem~\ref{thm-obstruction}]
Upon reducing $r$, we may without loss of generality assume that $B(x, r) \subset \subset \Omega$. Letting $K := \mathrm{supp}(\tar)$, we consider $\adjtar$ and $T^\star$ as given by Lemma~\ref{nice-adjoint}.

Let $T\leq T^\star$ be fixed. For any control $u \in E$, any $y_0,\tar \in L^2(\Omega)$, any solution to the adjoint equation~\eqref{backward} such that $p(T)= \adjtar$, we have $\textstyle \frac{d}{dt} \langle y(t), p(t)\rangle_{L^2} = \langle p(t), u(t)\rangle_{L^2}$.
As a result and owing to $y_0 = 0$,
\begin{equation}
\label{conflicting-inequality}
\langle y(T), p_f\rangle_{L^2} = \int_0^T \langle p(t), u(t)\rangle_{L^2}\, dt.\end{equation}

We now assume by contradiction that, for any $\e>0$ there exists a nonnegative control $u_\e \in E$ satisfying $\forall t \in (0,T), \, \mathrm{supp}(u_\e(t)) \cap B(x, r) = \emptyset$ and steering $y_0=0$ to the ball $\overline{B}(\tar, \e)$ in time $T$. We inspect the sign of the equality~\eqref{conflicting-inequality}
along the controls $u_\e$, $\e>0$.

On the one hand, because of the condition (ii) in Lemma~\ref{nice-adjoint} satisfied by $p$, and owing to $u_\e \geq 0$, the right-hand side of~\eqref{conflicting-inequality} is nonnegative, \ie \begin{equation}
\label{nonnegative}
\langle y(T), \adjtar \rangle_{L^2} \geq 0.
\end{equation}
On the other hand, the left-hand side of~\eqref{conflicting-inequality} satisfies 
\[\langle y(T), \adjtar \rangle_{L^2} =  \langle \tar, \adjtar \rangle_{L^2} + \langle y(T)-\tar, p_f\rangle_{L^2}  \leq\langle \tar, \adjtar \rangle_{L^2} +\e\|\adjtar\|_{L^2}\]
Now, $\langle \tar, \adjtar \rangle_{L^2} < 0$, because of (i) in Lemma~\ref{nice-adjoint}. As a result, there exists $\alpha>0$  such that $\adjtar \leq -\alpha$ on~$K$, so that 
\[\langle \tar, \adjtar \rangle_{L^2} \leq -\alpha \int_{K} \tar < 0,\]
because $\tar$ is nonnegative and nontrivial on $K$ by assumption.

Hence, for $\e>0$ small enough, $\langle y(T), \adjtar \rangle_{L^2} < 0$, which contradicts~\eqref{nonnegative}.
\end{proof}
\begin{rmrk}
As the proof shows, the obstruction to nonnegative approximate controllability in $\U_+$ does not rely on the comparison principle, but is of dual nature. As evidenced by the proof above, the core idea is indeed to construct $\adjtar$ and $\tar$ such that the equality~\eqref{conflicting-inequality} prevents $y(T)$ from being close to $\tar$.
The proof of Theorem~\ref{thm-obstruction} follows directly from the existence of $\adjtar$ satisfying the assumptions of Lemma~\ref{nice-adjoint}. Hence, this obstruction to nonnegative approximate controllability is rather general and will be satisfied by any operator (including uniformly second-order elliptic operators of the form~\eqref{def_elliptic}) for which such an element $\adjtar$ can be built. 
\end{rmrk}

\section{Further comments}
\label{sec-comments}
\subsection{Properties of the value function in the general case}
For general linear operators generating a $C_0$ semigroup,
%satisfying the hypothese of Theorem \ref{thm-controllability} and
fixing $\Omega$, $L$, $\e$, $y_0$ and $\tar$, we analyse the dependence with respect to the final time~$T$, for the optimal control problem \eqref{ocp-primal} studied in Section \ref{sec-controllability} for system \eqref{nonnegative-controllability}.

%We will denote the dual functional as follows:
%\begin{equation}\label{value-function-T}J_T(\adjtar):=F_T^\ast(L_T^\ast \adjtar)- \langle \tar - S_T y_0, \adjtar\rangle + \varepsilon \|\adjtar\|,\end{equation}
%where we have emphasised the dependence of the cost function $F$ with respect to the final time $T$.

%Recalling the notation
%\[\mathcal{M}(u)= \max\left(\|u\|_{\infty}, \frac{\|u\|_1}{\mL}\right),\]
By Lemma~\ref{lem-contact-condition} and Proposition~\ref{prop-coercive}, the optimal control problem~\eqref{ocp-primal} is well-posed, \ie optimal controls exist (see also Remark~\ref{rmk-general-existence}), hence we may consider %solved in Theorem \ref{prop-second-cost}:
\begin{equation}
    \label{norm-optimal-control-problem}
\Pi(T):= \frac{1}{2} (\overline{M}(T))^2:=\inf \{\cost(u), \quad u\in E, \quad \|L_T u-(\tar - S_T y_0)\|_{L^2}\leq \e\}, \quad T>0.\end{equation}
When $A^\ast$ satisfies the~\eqref{main_property} property and $\partial_t - A^\ast$  is analytic-hypoelliptic, $\overline{M}(T)$ is the amplitude of the unique optimal control in Proposition~\ref{prop-second-cost}.

Recall that by strong duality, we have
\begin{equation}
\label{duality-M-J}
\Pi(T) =\frac{1}{2} (\overline{M}(T))^2
=-\dualfct(p^\star_{T}),\quad \forall T\geq 0,
\end{equation}
where $p_T^\star$ is the unique minimiser of $\dualfct$. This is exactly the identity obtained for the HUM method where the cost functional $\cost$ is just $\textstyle \frac{1}{2}\|\cdot\|_E^2$. 

We first establish the continuity of $T \mapsto \overline{M}(T)$. 
\begin{prpstn}
 \label{prop-M-continuous}
 $\overline{M}$ (and thus $\Pi$) are continuous on $(0, +\infty)$.
 \end{prpstn}
 
 \begin{proof}
Using \eqref{duality-M-J}, we prove the continuity by showing that $(\adjtar, T) \mapsto \dualfct(\adjtar)$ (given by \eqref{dualfct}) satisfies the assumptions of Lemma~\ref{parametric_optimisation} with $H = L^2(\Omega)$ and $Z = (0,+\infty)$. Clearly, the first, second and fourth assumptions are satisfied, hence we are left with proving that 
$(\adjtar, T) \mapsto \dualfct(\adjtar)$ is weak-strong lower semicontinuous over $L^2(\Omega) \times (0,+\infty)$.
The last two terms of~\eqref{dualfct} are easily seen to be weak-strong lower semicontinuous over $L^2(\Omega) \times (0,+\infty)$, hence we investigate the property for the remaining term $F_T^\ast(L_T^\ast \adjtar)$.

Given $\adjtar \in L^2(\Omega)$ and $T >0$, let $(\adjtar^n)$ and $(T_n)$ be two sequences such that $\adjtar^n \rightharpoonup \adjtar$, $T_n \rightarrow T$.
We decompose 
\begin{align*}
F_{T_n}^\ast(L_{T_n}^\ast \adjtar^n)& = 
F_{T}^\ast(L_{T}^\ast \adjtar^n)+ \left(F_{T_n}^\ast(L_{T_n}^\ast \adjtar^n)-F_{T}^\ast(L_{T}^\ast \adjtar^n)\right).
\end{align*}
By weak (sequential) lower semicontinuity of $F_T^\ast$ over $L^2(0,T;L^2(\Omega))$, we find that the first term  satisfies 
\[F_{T}^\ast(L_{T}^\ast \adjtar) \leq \liminf_{n\rightarrow +\infty} F_{T}^\ast(L_{T}^\ast \adjtar^n).\]
To conclude, we only need to prove that the second term tends to $0$ as $n \rightarrow +\infty$. 

Using the notation $q_n$ for the solution to the forward adjoint problem such that $q_n(0)=\adjtar^n$, \ie $q_n(t) = S_t^\ast \adjtar^n$,
we have
\begin{align*} F_{T_n}^\ast(L_{T_n}^\ast \adjtar^n)-F_{T}^\ast(L_{T}^\ast \adjtar^n) 
& = \frac{1}{2} \left(\int_0^{T_n} \sigma_{\overline{\U}_L}(q_n(T_n-t)) \,dt\right)^2 - \frac{1}{2} \left(\int_0^{T}  \sigma_{\overline{\U}_L}(q_n(T-t))\right)^2 \\
& = \frac{1}{2}\left(\int_T^{T_n} \sigma_{\overline{\U}_L}(q_n(t))\,dt\right) \left(\int_0^{T_n} \sigma_{\overline{\U}_L}(q_n(t)) \,dt + \int_0^{T} \sigma_{\overline{\U}_L}(q_n(t)) \,dt \right)
\end{align*}
%= \int_T^{T_n} \sigma_{\overline{\U}_L}(q_n(t))\,dt

Using the bound $0 \leq \sigma_{\overline{\U}_L}(p) \leq|\Omega|^{1/2} \|p\|_{L^2}$ (see the proof of Lemma~\ref{lem-contact-condition}) and the estimate $\|S_t\|_{L(L^2(\Omega))} \leq  C$ valid for all $t \in [0,T+1]$ with $C>0$ some constant independent of $n$, we have
\begin{align*}
\left|\int_0^{T_n} \sigma_{\overline{\U}_L}(q_n(t)) \,dt + \int_0^{T} \sigma_{\overline{\U}_L}(q_n(t)) \,dt\right| \leq C |\Omega|^{1/2} (T+T_n)
 \, \|\adjtar^n\|_{L^2},
\end{align*}
a bounded quantity, and
\begin{align*}
\left|\int_T^{T_n} \sigma_{\overline{\U}_L}(q_n(t)) \, dt\right| \leq C |\Omega|^{1/2} |T-T_n|
 \, \|\adjtar^n\|_{L^2},
\end{align*}
which tends to $0$ as $n\rightarrow +\infty$.
\end{proof}

We now study the behaviour of $\overline{M}(T)$  near $T=0$ and $T=+\infty$. We recall that $\overline{M}(T)$ also depends on all other parameters $y_0$, $\tar$, $\e$ and $L$.

We now recall (see \cite{pazy2012semigroups}) that there exist $C_s>0$, $\alpha \in \R$ such that for all $t \geq 0$, $\|S_t\|_{L(L^2(\Omega))} \leq C_s e^{\alpha t}$,
and the semi-group generated by $(A, D(A))$ is said to be \textit{exponentially stable} if $\alpha<0$.

\begin{prpstn} 
\label{prop-lower-bound}
We have 
\begin{equation}
\label{lower_bound}
\forall T>0, \quad \overline{M}(T) \geq |\alpha| \frac{\|\tar - S_T y_0\|_{L^2} - \e}{\sqrt{\mL} (1-e^{\alpha T})}.
\end{equation}
%\[\mu_+=\lim_{T\to 0}M^\ast(T)=+\infty.\]
\end{prpstn}

\begin{proof}
Let $u^\star_T$ be an optimal control in time $T$ for the optimal control problem \eqref{norm-optimal-control-problem}, then
\begin{equation*}
    %\label{a-priori-estimate}
    {\def\arraystretch{2}\begin{array}{rl}
\|L_T u^\star_T\|_{L^2}& =\left\|\int_0^T S_{T-t} u^\star_T(t, \cdot)dt\right\|_{L^2} 
\leq\int_0^T  \|S_{T-t} u^\star_T(t, \cdot)\|_{L^2} dt \\
& \leq \int_0^T  e^{\alpha(T-t)} \|u^\star_T(t, \cdot)\|_{L^2} dt 
\leq \frac{1}{|\alpha|}(1- e^{\alpha T}) \overline{M}(T) \sqrt{\mL}.
\end{array}}\end{equation*}
Now, by definition of our control problem, for all $T>0$, $\|\tar - S_T y_0\|_{L^2}-\e \leq \|L_T u^\star_T\|_{L^2}$,
and the result follows.
\end{proof}

\begin{crllr}
\label{prop-M(t)-limits}
Assume that $\tar\notin \overline{B}(y_0, \e)$. Then:
\begin{equation}\label{value-asymp-near-0}\frac{1}{T} = \underset{T \rightarrow 0}{O} (\overline{M}(T)).\end{equation}
In particular, $\overline{M}(T)\xrightarrow[T\to 0]{}+\infty$.

Assume that $\tar\notin \overline{B}(0, \e)$. If, additionally, $(S_t)_{t\geq 0}$ is exponentially stable, then
\begin{equation}\label{value-asympt-infty}\liminf_{T \rightarrow +\infty} \overline{M}(T) > 0.\end{equation}

\end{crllr}
\begin{proof}
The estimate \eqref{value-asymp-near-0} is obtained by passing to the limit in~\eqref{lower_bound}, using that $S_T y_0\xrightarrow[T\to 0]{}y_0$: the lower bound behaves as $\textstyle \frac{\|\tar -y_0\|_{L^2} -\e}{\sqrt{\mL}}\; \frac{1}{T}$.
The inequality \eqref{value-asympt-infty} is obtained by passing to the limit $T \rightarrow +\infty$ in~\eqref{lower_bound}, using that $S_T y_0 \xrightarrow[T\to \infty]{}0$:
\[\liminf_{T \rightarrow +\infty} \overline{M}(T) \geq |\alpha|\frac{\|\tar\|_{L^2} -\e}{\sqrt{\mL}}>0.\]
\end{proof}

\subsection{Obstructions}
We further investigate the behaviour of $\overline{M}$, and establish results on the corresponding minimal time problem \eqref{time-optimal-control-problem}.
%, we will now focus on the case of uniformly elliptic operators of the form \eqref{def_elliptic} (see Section \ref{subsec-results}). 
The comparison principle formulated in \eqref{comp_principle} will be a key ingredient in our study.

\subsubsection{Obstruction to reachability and small-time controllability}

Given the controllability result of Theorem \ref{thm-controllability}, in order to study possible obstructions, we introduce a new bound on the amplitude of the control, of the form:
\begin{equation}
    \label{upper_constraint_bis}
    M(u):=2\sqrt{\cost(u)} \leq M_{\max}, \quad u\in E,
\end{equation}
for some $M_{\max}>0$. Note that such a constraint imposes nonnegativity of the control. With this new constraint on the controls, we illustrate a general property that is well known for finite-dimensional systems: exponential stability prevents reachability.

In particular, the result below holds for uniformly elliptic operators of the form \eqref{def_elliptic} with $0$th order coefficient satisfying $c \leq 0$.

\begin{prpstn}
\label{prop-bound_amplitude}
Assume that $(S_t)_{t \geq 0}$ is exponentially stable.
Let $(y_0,\tar)$ be such that for all $T \geq 0$, $\tar \geq S_T y_0$ and $\|S_T y_0- \tar\|_{L^2} \geq \delta$ for some $\delta>0$.
%$tar\geq S_T y_0$, $y_f \notin \overline{B}(S_T y_0, \e)$ and $y_f \notin \overline{B}(0,\e)$, 
Then, for all $0<\e<\delta$ there exists $M_{\max}{M}(y_0, \tar, \e)>0$ satisfying 
\begin{itemize}
    \item if $M_{\max}>  M_{\max}(y_0, \tar, \e)$, there exists a time $T>0$ and a control $u \in E$  satisfying~\eqref{upper_constraint_bis}, steering $y_0$ to $\overline{B}(\tar,\e)$ in time $T$.
    If $A^\ast$ satisfies the~\eqref{main_property} property and $\partial_t - A^\ast$ is analytic-hypoelliptic, the control may be chosen to be in $\U^L_{\textrm{shape}}$.
    \item
    if $M_{\max}< M_{\max}(y_0, \tar, \e)$, no such control exists.
\end{itemize}

Moreover, for all $M_{\max}>0$, the control system \eqref{nonnegative-controllability} is not nonnegatively approximately controllable with controls in $\{M(u)\leq M_{\max}\}$ in any time $T>0$.
%\cap \U^L_{\textrm{shape}}$.

%For any $\e>0$, $\tar \in L^2(\Omega)$, there exists $\mu_-(\tar, \e)>0$ such that 
%\[\|L_T u - \tar \|>\e, \quad \forall T>0, \quad \forall u \in E, \quad F(u) \leq \mu_-(\tar, \e).\]
\end{prpstn}
\begin{proof}
Given Corollary \ref{prop-M(t)-limits}, the function $\overline{M}(T)$ goes to $+\infty$ as $T \rightarrow 0$, is bounded away from~$0$ at infinity, and does not vanish over the interval $(0,+\infty)$. Since it is continuous, we define 
\[M_{\max}(y_0, \tar, \e):= \inf_{T > 0} \overline{M}(T)>0,\]
and the first two claims follow. When $A^\ast$ satisfies the~\eqref{main_property} property and $\partial_t - A^\ast$ is analytic-hypoelliptic, the control may be chosen to be in $\U^L_{\textrm{shape}}$ by Theorem~\ref{thm-controllability}.

Then, let $M_{\max}>0$. Taking $\tar\in L^2(\Omega)$ such that $\|\tar\|_{L^2}> \textstyle\frac{\sqrt{\mL}}{|\alpha|}M_{\max}+\e$
and $y_0\in L^2(\Omega)$ such that $\tar \geq S_T y_0$ and $\|S_T y_0- \tar\|_{L^2} \geq \delta>0$. Thanks to the proof of Corollary~\ref{prop-M(t)-limits}, we infer $M_{\max}(y_0, \tar, \e) \geq |\alpha| \textstyle \frac{\|\tar\|_{L^2}-\e}{\sqrt{\mL}}>M_{\max}$.
It follows from the second claim that $y_0$ cannot be steered to $\tar$ in any time $T>0$ with a control $u$ such that $M(u)\leq M_{\max}$.
Thus, system \eqref{nonnegative-controllability} is not nonnegatively approximately controllable with such controls in any time $T>0$.

\end{proof}

 \subsubsection{Characterisation of minimal time controls}

Throughout this section, we let $\e>0$, $\tar\in L^2(\Omega)$, we assume that \eqref{non-trivial-target} holds, and let $y_0=0$.
Hence we must have $\|\tar\|_{L^2} > \e$ and the condition \eqref{non-trivial-target} is independent of $T$. Finally, $\tar \geq S_T y_0$ here simply amounts to $\tar \geq 0$.

Given the obstruction result of Proposition \ref{prop-bound_amplitude}, we consider the \textit{minimal time} control problem:
\begin{equation}
    \label{time-optimal-control-problem}
    T^\star(\lambda)=\inf\{T>0, \quad \exists u\in E,  \quad \|L_T u-\tar\|_{L^2}\leq \e, \quad \cost(u)\leq \lambda\}, \quad \lambda>0.
\end{equation}

From our study of the optimal control problem~\eqref{ocp-primal}, we know that this minimal time is well defined for $\lambda\in \overline{M}((0,+\infty))$. Under appropriate assumptions, we will show that it is reached, and characterise the minimal time controls, by establishing a form of equivalence between the optimal control problem and the corresponding minimal time problem. This is now a well-known feature for parabolic equations (see~ \cite{wang_equivalence_2012, Kunisch-Wang-2013, qin2018equivalence}).

 % We now turn to the limits of $M^\ast$. As it is defined on $(0, +\infty)$, continuous, non-increasing, and non-negative, we deduce that
 %\begin{itemize}
 %    \item it either reaches a finite limit at $0$, or tends to $+\infty$.
 %    \item it decreases to a non-negative limit at $+\infty$.
 %\end{itemize}
 
\paragraph{Further study of the value function $\overline{M}$.}

 Using strong duality again, we will establish that $\overline{M}$ is a non-increasing function under the assumption that $A^\ast$ satisfies the comparison principle~\eqref{comp_principle}. We start with the following general lemma:
 \begin{lmm}
\label{lem-J-comparison}
Given any $0<T_1<T_2$, and $y_0=0$, for a general unbounded linear operator $A$, the dual functional defined by~\eqref{dualfct} satisfies:
\begin{equation}
    \label{value-function-decreasing}
    J_{T_1, \e}(\adjtar) \leq J_{T_2,\e}(\adjtar), \quad \forall \adjtar\in L^2(\Omega),
\end{equation}
 with equality if and only if 
 \begin{equation}\label{value-function-equality}
     L_{T_2}^\ast \adjtar (t) \leq 0, \quad \forall t \in [0, T_2- T_1].
 \end{equation}
\end{lmm}
\begin{proof}
Since $y_0=0$, inequality \eqref{value-function-decreasing} follows immediately from the comparison of the integral terms in the expression of the $J_{T_i,\e}, i\in \{1,2\}.$
Moreover, for $\adjtar\in L^2(\Omega)$, one has
$J_{T_1,\e}(\adjtar)=J_{T_2,\e}(\adjtar)$
if and only if
\[\int_0^{T_1} \sigma_{\overline{\U}_L}\left(L_{T_1}^\ast \adjtar (t)\right) dt=\int_0^{T_2} \sigma_{\overline{\U}_L}\left(L_{T_2}^\ast \adjtar (t)\right) dt,\]
that is, by definition of the operators $L_{T_i}^\ast$ (see \eqref{backward} which are obviously related by $L_{T_1,\e}^\ast \adjtar (t)  =L_{T_2,\e}^\ast \adjtar (T_2-T_1+t)$ for all $t \in(0,T_1)$,
\[\int_{0}^{T_2-T_1}\sigma_{\overline{\U}_L}\left(L_{T_2}^\ast \adjtar (t)\right) dt=0.\]
Using the definition of the support function $\sigma_{\overline{\U}_L}$ (see the proof of Lemma~\ref{continuation}), this is equivalent to~\eqref{value-function-equality}.
\end{proof}

\begin{crllr}
\label{decreasing}
The function $\overline{M}$ (and hence $\Pi$) are non-increasing on $(0, +\infty)$.
\end{crllr}

 We now denote 
$\mu_-= \mu_-(\tar):= \lim_{T\rightarrow +\infty} \Pi(T) = \lim_{T\rightarrow +\infty} \textstyle \frac12 \overline{M}(T)^2$. 
Note that $\mu_- \in [0,+\infty)$, and if the semi-group generated by $A$ is exponentially stable, $\mu_->0$ as established by~\eqref{value-asympt-infty} in Corollary~\ref{prop-M(t)-limits}.

 \begin{prpstn}
%Assume $A$ is a uniformly elliptic operator of the form \eqref{def_elliptic}. 
Assume $A^\ast$ satisfies the comparison principle~\eqref{comp_principle}.
Then, there exists $T_\ell = T_\ell(\tar) \in (0, +\infty]$ such that $\overline{M}$ is decreasing on $[0, T_\ell)$, and constant on $[T_\ell, +\infty)$.
\end{prpstn}
\begin{rmrk}
The proposition above implies in particular that 
$\overline{M}$ either decreases on the whole of $(0, +\infty)$ to its limit $\mu_-$ (if $T_\ell = +\infty)$, or reaches it at $T_\ell<+\infty$ and then remains constant. 
\end{rmrk}
\begin{proof}By strong duality, Lemma \ref{lem-J-comparison} implies that $\overline{M}$ is non-increasing. 
Let $T_2>T_1>0$, and denote %$u^\star_{T_1}, u^\star_{T_2},
$p^\star_{T_1}, p^\star_{T_2}$ the associated dual minimisers. Assume that 
\begin{equation}\label{M-flat}
   \overline{M}(T_1)=\overline{M}(T_2).
\end{equation}
From Lemma \ref{lem-J-comparison}, and by definition of $p_{T_1}^\star$, we know that 
\begin{equation}\label{J-ineq-chain}J_{T_1,\e}(p_{T_1}^\star)\leq J_{T_1,\e}(p_{T_2}^\star) \leq J_{T_2,\e}(p_{T_2}^\star).\end{equation}
From \eqref{duality-M-J}, \eqref{M-flat} implies that 
$J_{T_1,\e}(p_{T_1}^\star)=J_{T_2,\e}(p_{T_2}^\star),$
so that all the inequalities in \eqref{J-ineq-chain} actually are equalities.

By uniqueness of the dual optimal variable (Proposition \ref{prop-uniqueness-p-u}), the first equality implies that
\begin{equation}\label{common-p}
    p_{T_1}^\star=p_{T_2}^\star=:\adjtar^\star.
\end{equation}
From Lemma \ref{lem-J-comparison}, the second equality implies that
\begin{equation}\label{p-becomes-negative}
    L_{T_2}^\ast \adjtar^\star (t) \leq 0, \quad \forall t\in [0, T_2-T_1].
\end{equation}
From \eqref{common-p} and \eqref{p-becomes-negative}, we get 
$p_T^\star = \adjtar^\star$ for all $T \in [T_1, T_2]$.
Now, for $T>T_2$, the comparison principle~\eqref{comp_principle} and inequality \eqref{p-becomes-negative} imply that 
$L_{T}^\ast \adjtar^\star (t) \leq 0$ for all  $t\in [0, T-T_1].$
From Lemma \ref{lem-J-comparison}, we then get
$J_{T,\e}(\adjtar^\star)=J_{T_1,\e}(\adjtar^\star), $
which implies
$J_{T,\e}(\adjtar^\star)=J_{T_1,\e}(\adjtar^\star) \leq J_{T,\e}(p_T^\star).$
By definition of the dual minimiser~$p_T^\star$ of $J_{T,\e}$, we also have $J_{T,\e}(p_T^\star)\leq J_{T,\e}(\adjtar^\star),$
and then finally, 
$J_{T,\e}(p_T^\star)=J_{T,\e}(\adjtar^\star),$
\ie 
    $
    p_T^\star=\adjtar^\star.
$
This implies, thanks to \eqref{duality-M-J}, that
    $    \overline{M}(T)=\overline{M}(T_1)=\overline{M}(T_2),$
which proves the proposition.
\end{proof}
\begin{rmrk}
It follows from all the above and \eqref{p-becomes-negative} that, when $A^\ast$ satisfies the comparison principle~\eqref{comp_principle}, if $T_\ell<+\infty$, then
\begin{equation*}
%\label{p-becomes-negative-T-l}
    L_{T}^\ast p_{T_\ell}^\star (t) \leq 0, \quad \forall T\geq T_\ell, \quad \forall t\in [0, T-T_\ell],
\end{equation*}
and
\begin{equation*}
%\label{common-optimal-control-T-l}
    u_T^\star(t)=\left\{\begin{aligned}
   & 0 & \quad \text{if} \quad t\in (0, T-T_\ell), \\
&    u_{T_\ell}^\star(t-T+T_\ell)& \quad \text{if} \quad t\in( T-T_\ell,T),
    \end{aligned}\right., \quad \forall T\geq T_\ell
\end{equation*}
is an optimal control on $[0,T]$ whenever $u_{T_\ell}$ is an optimal control on $[0, T_\ell]$.
\end{rmrk}

We now establish the relationship between the optimal control problem \eqref{norm-optimal-control-problem} and the minimal time control problem.
\begin{prpstn}\label{prop-norm-time-equiv}
%Assume $A$ is a uniformly elliptic operator of the form \eqref{def_elliptic}. 
Assume that $A^\ast$ satisfies the comparison principle~\eqref{comp_principle}. Then, for all $T\in (0, T_\ell)$, any optimal control for \eqref{norm-optimal-control-problem} on $[0,T]$ is a minimal time control, that is,
\begin{equation*}
    \label{T*-M*}
    T^\star(\Pi(T))=T.\end{equation*}
Moreover, for any $\lambda>\mu_-$,
\begin{equation*}
    %\label{M*-T*}
    \Pi(T^\star(\lambda))=\lambda.\end{equation*}
\end{prpstn}
 \begin{proof}
 We proceed by contradiction. Assume that $T^\star(\Pi(T))<T.$
 Then, there exists $\delta>0$ and a control $u_\delta\in L^2(0,T-\delta;L^2(\Omega))$ such that
 $\cost(u_\delta)\leq \Pi(T).$
 Now, any optimal control $u^\star_\delta$ (in the sense of optimal control problem \eqref{norm-optimal-control-problem} in time $T-\delta$) satisfies $\cost(u^\star_\delta)\leq \cost (u_\delta)$  (the inequality is not necessarily strict, as $u_\delta$ could be an optimal control),
 \ie
$\Pi(T-\delta)=\cost(u^\star_\delta)\leq \cost(u_\delta) \leq \Pi(T),$
 which contradicts the fact that $T\mapsto \Pi(T)$ is a decreasing function on $(0, T_\ell)$. Thus, \eqref{T*-M*} holds.
 
 Now, let $ \lambda>\mu_-$. From Corollaries \ref{prop-M(t)-limits}, \ref{decreasing} and Proposition \ref{prop-M-continuous}, there exists $T\in(0, T_\ell)$ such that 
 $\Pi(T)=\lambda.$
 Applying $T^\star$ to the above and using \eqref{T*-M*}, we get
$T^\star(\lambda)=T^\star(\Pi(T))=T.$
 Then, applying $\Pi$ to the above yields
$\Pi(T^\star(\lambda))=\Pi(T)=\lambda.$ \end{proof}

We can also formulate the above result in the following way: for all $\lambda > \mu_-$,
 \[T^\star(\lambda)=\inf\{T>0, \quad \Pi(T)\leq \lambda\},\]
 that is, $T^\star$ is the pseudo-inverse of $\Pi$ on $(\mu_-, +\infty)$.

 %Given the decreasing nature of $M^\ast(T)$, we immediately get the following corollary:
% \begin{crllr} 
 %\label{cor-limit-M*-mu}
% There exist $\mu_+ \in (0,+\infty], \mu_- \in [0,+\infty)$ (dependent on $\mL, y_f, \Omega, \e)\geq 0$)
 %such that
 %\[\lim_{T \rightarrow 0} M^\ast(T) = \mu_+, \quad \lim_{T \rightarrow +\infty} M^\ast(T) = \mu_-.\]
 %and $T^\star$ is defined and strictly decreasing on $(\mu, +\infty)$. 
 
 %Moreover for $\lambda >\mu$, there exists a time optimal control for the time-optimal control problem $T^\star(\lambda)$.
 % \end{crllr}

%\subsection{Study of the time optimal control problem}

In terms of the time optimal control problem, we now have a complete characterisation of time optimal controls for \eqref{time-optimal-control-problem}: 
 \begin{thrm}
Assume that $A^\ast$ satisfies the comparison principle~\eqref{comp_principle}.
 \label{thm-existence-time-optimal-controls}
  %Assume $A$ is a uniformly elliptic operator of the form \eqref{def_elliptic}. 
 For any $\lambda>\mu_-$, $T^\star(\lambda)<+\infty$, and $T^\star(\lambda)\xrightarrow[\lambda\to\infty]{}0$, $T^\star(\lambda)\xrightarrow[\lambda\to\mu_-]{}+\infty$.
 As a consequence, the domain of definition of $T^\star$ is $(\mu_-,+\infty)$, and on its domain of definition, $T^\star$ is continuous and decreasing.
 
Moreover, if $A^\ast$ satisfies the~\eqref{main_property} property and $\partial_t - A^\ast$ is analytic-hypoelliptic, there exists a unique minimal time control for \eqref{time-optimal-control-problem}, given by the optimal control problem \eqref{ocp-primal}, and it lies in $\U^L_{\textrm{shape}}$ .
 \end{thrm}

%\section{Further questions}

%The general formulation of the method of proof applied here will be treated in an upcoming work.

\paragraph{Acknowledgments.} The authors are grateful to Rémy Abergel for enlightening discussions about Fenchel duality. All three authors acknowledge the support of the ANR project TRECOS, grant number ANR-20-CE40-0009. 

\appendix
\section{Convex analysis}
\label{app-convex-analysis}
\subsection{Core properties of Fenchel conjugation}
\label{app-subsec-conj}
A fundamental property of conjugation is involution (over $\Gamma_0(H))$:
\begin{thrm}[Fenchel-Moreau]
Given any $f\in\Gamma_0(H)$, there holds $f^\ast \in \Gamma_0(H)$ and $f^{\ast \ast} = f$.
\end{thrm}

Analogously to the classical gradient, the subdifferential can be used to study optimality:
\begin{prpstn}[Fermat's rule]
\label{max-property}
Let $f\in \Gamma_0(H)$. $f$ attains a finite global minimum over $H$ in $x^\star$ if and only if 
\begin{equation*}
    0\in \partial f(x^\star).
\end{equation*}
\end{prpstn}

We now list further useful properties of the Fenchel conjugate:
\begin{itemize}
\item multiplication by a real number: for $\alpha \in \R$,
\begin{equation}\label{alpha-conjugate}
(\alpha f)^\ast(y)= \left\{
    \begin{aligned}
        \alpha f^\ast \left(\frac{y}{\alpha}\right) \qquad &\textrm{if} \quad \alpha \neq 0,\\
      \sigma_{\operatorname{dom} (f)}(y) \qquad  &\textrm{if} \quad \alpha=0.
    \end{aligned}\right.\end{equation}
%On en déduit en particulier pour $x_0 \in E$, $\e>0$ l'identité 
\item the (suitably normalised) squared norm is its own conjugate:
\begin{equation}
    \label{square-conjugate}
    \left(\frac{1}{2} \|\cdot\|_H^2 \right)^\ast  = \frac{1}{2}\|\cdot\|_{H}^2.
\end{equation}
\end{itemize}

%\subsection{Important properties}

Let us also mention a result about composition~\cite{hiriart2006note}.
First, let $f \in \Gamma_0(H)$ and  $g \in \Gamma_0(\R)$ be non-decreasing. Then,
\begin{equation*}
%\label{composition}
(g\circ f)^\ast(y)= \min_{\alpha \geq 0} \left(g^\ast(\alpha)+\alpha f^\ast\Big(\frac{y}{\alpha}\Big)\right).\end{equation*}
Following \eqref{alpha-conjugate}, the convention for $\alpha = 0$ is $0\, f^\ast\Big(\frac{y}{0}\Big) = \sigma_{\mathrm{dom}(f)}(y).$

\paragraph{Link with the subdifferential.} We now give another characterisation of the subdifferential set, which illustrates the link with convex conjugation: for $f\in \Gamma_0(H)$,
\begin{equation}
\label{subdiff-and-convex-conjugate}
\begin{aligned}
\partial f(x) = \{p \in H, \, \langle p,x \rangle_H - f(x)= f^\ast (p)\}  = \{p\in H, \, \langle x, p\rangle_H -f^\ast(p)=f(x)\}
\end{aligned}
\end{equation}
Essentially, the subdifferential is the set of linear forms on which the convex conjugate is attained.

Using this characterisation, we then get the Legendre-Fenchel identity, which allows us to ``flip'' subdifferentials:
\begin{equation}
    \label{flip-subdiff}
    p\in \partial f(x)
    %\iff \langle x, p\rangle -f(x)=f^\ast(p)  \iff \langle p,x \rangle - f^\ast(p)=f(x) 
    \iff x\in \partial f^\ast(p), \quad f\in \Gamma_0(H),\; \forall x, p \in H.
\end{equation}

\subsection{Some properties of indicator and support functions}
\label{subsec-app-indic}
Indicator functions are a crucial tool to encode constraints in convex optimisation problems. Their properties are closely linked to topological properties of their indicated sets:
\begin{prpstn}
We have $\delta_C, \sigma_C \in \Gamma_0(H)$ as soon as $C$ is non-empty, convex and closed.
\end{prpstn}

The characterisation \eqref{subdiff-and-convex-conjugate} of the subdifferential yields a useful result on indicator functions:
\begin{prpstn}
\label{prop-indicator-subdiff}
Let  $C\subset H$ be a closed convex set with nonempty interior. Then, for $x\in H$ we have the following:
\[x\in \partial C \quad \iff \quad \partial \delta_C (x) \ \textrm{is a nontrivial cone}.\]
Equivalently, by convex conjugation,
\[\exists p  \neq 0, \; x\in \argmax_{v\in C} \langle v, p\rangle  \ \iff \exists p  \neq 0, \;\ x\in \partial \sigma_C (p) \ \iff \ x\in \partial C.\]
\end{prpstn}

\paragraph{Indicator function of a ball in a Hilbert space.} Consider the closed unit ball $\overline{B}(0,1)$ of $H$. We have seen before that 
\[ \sigma_{\overline{B}(0,1)}(y) = \big(\delta_{\overline{B}(0,1)}\big)^\ast(y)  = \|y\|_H.\]
Using \eqref{subdiff-and-convex-conjugate}, we get the following: for $x\in \overline{B}(0,1)$,
\begin{equation*}
    \partial \delta_{\overline{B}(0,1)}(x)=\{p\in H, \quad \langle p,x\rangle_H=\sigma_{\overline{B}(0,1)}( p)\}=\{p\in H, \quad \langle p,x\rangle_H=\|p\|_H\}.
\end{equation*}
From the Cauchy-Schwarz inequality we know that 
$\langle p, x \rangle_H \leq \|p\|_H \|x\|_H,$
it follows that 
$\langle p, x \rangle_H =\|p\|_H$ if and only if 
$x=\textstyle{\frac{p}{\|p\|_H}}$.
This implies that
\begin{equation}
\label{subdiff-indic}
    \partial \delta_{\overline{B}(0,1)}(x)=\left\{
    \begin{array}{cl}
        \{0\} &\qquad\textrm{if} \quad \|x\|_H<1,\\
        \{\lambda x , \quad \lambda \geq 0 \} &\qquad \textrm{if} \quad \|x\|_H=1.
    \end{array}
\right.
\end{equation}

\subsection{Technical lemmas}
\label{app-subsec-lemmas}
%We now list some technical lemmas which are crucial for our analysis. 
%We first give a useful result to compute conjugates and subdifferentials of some functions $f \in L^2(0,T;H)$.
\begin{lmm}\label{lem-integral}
Let $f \in \Gamma_0(H)$ be such that \[F : u \in L^2(0,T;H) \longmapsto \int_0^T f(u(t)) \,dt,\] is well-defined and proper. Then $F \in \Gamma_0(L^2(0,T;H))$, and its Fenchel conjugate and subdifferential are given by
\[\forall p \in L^2(0,T;H), \quad F^\ast(p) = \int_0^T f^\ast(p(t))\,dt,\]
\[\partial F(u) = \left\{p\in L^2(0,T;H), \quad p(t)\in \partial f(u(t)), \ \text{for a.e. } t\in (0,T)\right\}, \quad \forall u\in L^2(0,T;H).\]
\end{lmm}

\begin{proof}
Since $F$ is obviously convex, we only need to justify that $F$ is lsc to infer $F  \in \Gamma_0(L^2(0,T;H))$. 
We let $u_n \rightarrow u$ be in $L^2(0,T;H)$ and must show that $F(u)\leq \liminf F(u_n)$. 
Upon extraction of a subsequence, we may assume that $F(u_n) \rightarrow \liminf F(u_n)$, and that $u_n(t) \rightarrow u(t)$ in $H$ for a.e. $t \in (0,T)$. 
Then, using successively the lsc of $f$ and Fatou's lemma, we find
\[F(u) = \int_0^T f(u(t))\,dt \leq \int_0^T \liminf f(u_n(t))\,dt \leq\liminf \int_0^T  f(u_n(t))\,dt = \liminf F(u_n).\]
For $p \in L^2(0,T;H)$, we compute
\[
F^\ast(p) = \sup_{u \in L^2(0,T;H)} \langle p, u\rangle_{ L^2(0,T;H)} - \int_0^T f(u(t)) \, dt  
= \int_0^T \left(\sup_{u \in H} \langle p(t), u\rangle_H -f(u(t))\right) \,dt = \int_0^T f^\ast(p(t))\,dt.
\]

Using the characterisation given in \eqref{subdiff-and-convex-conjugate}, and Lemma~\ref{lem-integral}, we have the following:
\begin{equation*}
    \begin{aligned}
    \partial F(u)&=\argmax_{p\in L^2(0,T;H)} \{\langle p,u \rangle-F^\ast(p)\}\\
    &=\argmax_{p\in L^2(0,T;H)} \left\{\int_0^T \langle p(t),u(t)\rangle dt-\int_0^T f^\ast(p(t))dt\right\} \\
    &= \argmax_{p\in L^2(0,T;H)} \left\{\int_0^T \left( \langle p(t),u(t)\rangle - f^\ast(p(t)) \right)dt\right\} \\
    &=\left\{p\in L^2(0,T;H), \quad p(t)\in \argmax_{p\in H}\{\langle p , u(t)\rangle - f^\ast(p)\}\right\},
    \end{aligned}
\end{equation*}
and the result follows by the same characterisation of the subdifferential set $\partial f(u(t))$.
\end{proof}

\subsection{Fenchel-Rockafellar duality}\label{app-subsec-fenchel}
Let $E$ and $F$ be two Hilbert spaces. Let $f$ and $g$ be functions in $\Gamma_0(E)$ and $\Gamma_0(F)$, respectively, and $A : E \rightarrow F$ be a bounded operator. 
Consider the (primal) optimisation problem 
\begin{equation}\label{abstract-primal} \pi = \inf_{ x \in E} \left(f(x) + g(Ax)\right).\tag{$\mathcal{C}$}\end{equation}
and its dual problem 
\begin{equation}\label{abstract-dual} d= \sup_{z \in F}\left(- f^\ast(A^\ast z) - g^\ast(-z)\right) = -\inf_{ z \in F}\left(f^\ast(A^\ast z) + g^\ast(-z)\right) \tag{$\mathcal{D}$}
\end{equation}
With the above notations, weak duality always holds, \ie we always have $\pi\geq d$.
The Fenchel-Rockafellar theorem states when and how the strong duality holds, \ie when $d=\pi$~\cite{Rockafellar1967}.
\begin{thrm}
If there exists $\bar x \in E$ such that $g$ is continuous at $A\bar x$ and $f(\bar x) < +\infty$, then
\[\pi = d \quad  \text{and} \quad d \text{ is attained \textit{if finite}}.\]
Symmetrically, if there exists $\bar z \in F$ such that $f^\ast$ is continuous at $A^\ast\bar z$ and $g^\ast(-\bar z)<+\infty$, then 
\[d=\pi \quad  \text{and} \quad \pi \text{ is attained \textit{if finite}}.\]
\end{thrm}
The second part of the theorem is obtained by applying the first part to \eqref{abstract-dual}, $\inf_{ z \in F}\left(f^\ast(A^\ast z) + g^\ast(-z)\right),$
seen as a primal problem, and \eqref{abstract-primal}, rewritten as
$\sup_{ x \in E} \left(-f(x) - g(Ax)\right),$
seen as its dual problem.
This yields
$-d \geq -\pi,$
with equality under the corresponding assumptions.

\paragraph{Lagrangian and saddle-point interpretation.} Let us now define the Lagrangian for $(x,y) \in E \times F$ by
\[\mathcal L(x,y) := \langle y, A x\rangle + f(x) - g^\ast(y). \]
If $(x^\star, y^\star)$ is a saddle point of the Lagrangian, \ie
\begin{equation*}
%\label{saddle-point-def}
x^\star \in \argmin_{x \in E} \mathcal{L}(x, y^\star) \; \text{ and } \; y^\star \in \argmax_{y \in F} \mathcal{L}(x^\star, y),\end{equation*}
then $(x^\star, z^\star)$ (with $z^\star = -y^\star$) is a pair of primal and dual optimal variables, and strong duality holds.

What matters is the converse: if $(x^\star, z^\star)$ is a pair of primal and dual optimal variables and if strong duality holds, then $(x^\star, y^\star)$ (with $y^\star = -z^\star$) is a saddle point of $\mathcal L$.

%by definition of $y^\star$ and the convex conjugate, \eqref{saddle-point-def} implies
%\[\mathcal{L}(x^\star,y^\star)=  f(x^\star)+\max_{y\in F} \langle y, A x^\star\rangle - g^\ast(y) =  f(x^\star)+g(Ax^\star),\]
%and, symmetrically,
%\[\begin{aligned}
%\mathcal{L}(x^\star,y^\star)&=-g^\ast(y^\star) + \max_{x\in E} \langle y^\star, Ax\rangle + f(x) \\
%&=-g^\ast(y^\star) + \max_{x\in E} \langle A^\ast y^\star, -x\rangle + f(-x) \\
%&=-g^\ast(y^\star) - \min_{x\in E} \langle A^\ast y^\star, x\rangle - f(-x) \\
%&=-g^\ast(y^\star) - f^\ast(-A^\ast y^\star) \\
%&=-g^\ast(-z^\star) - f^\ast(A^\ast z^\star). \\
%\end{aligned}\]
%Hence, strong duality holds and $\pi$ (resp. $d$) is attained in $x^\star$ (resp. $z^\star$).

%%
Whenever $(x^\star, y^\star)$ is a primal-dual optimal pair, Fermat's rule and the Legendre-Fenchel identity yield 
\[x^\star \in \argmin_{x \in E} \mathcal{L}(x, y^\star) \quad  \iff \quad -A^\ast y^\star \in \partial f(x^\star)  \quad  \iff \quad x^\star \in \partial f^\ast(-A^\ast y^\star), \]
as well as
\[y^\star \in \argmax_{y \in F} \mathcal{L}(x^\star, y) \quad  \iff \quad  A x^\star  \in   \partial g^\ast(y^\star)  \quad  \iff \quad y^\star \in \partial g(Ax^\star), \]
Summing up, we have the following proposition:
\begin{prpstn}\label{saddle-point-stuff-weak}
Let $(x^\star, z^\star)$ be a pair of primal and dual optimal variables. If strong duality holds, then
\begin{equation}\label{saddle-x-identities}
    x^\star \in \partial f^\ast(A^\ast z^\star), \quad Ax^\star \in \partial g^\ast(-z^\star),
\end{equation}
\begin{equation*}
%\label{saddle-z-identities}
    z^\star \in -\partial g(Ax^\star), \quad A^\ast z^\star \in \partial f(x^\star).
\end{equation*}
\end{prpstn}
Reinterpreting the Fenchel-Rockafellar theorem with the above and in a way that is useful for controllability issues, we end up with 
\begin{prpstn}\label{saddle-point-stuff}
Under the assumption that there exists $\bar x \in E$ such that $g$ is continuous at $A\bar x$ and $f(\bar x)<+\infty$, if $\pi$ is finite, and attained at $x^\star \in E$, then $d$ is attained at $z^\star\in F$ satisfying
\begin{equation}\label{primal-to-dual-variable}
     z^\star \in -\partial g(Ax^\star), \quad A^\ast z^\star \in \partial f(x^\star).
\end{equation}
Conversely,  if $(x^\star, z^\star)$ satisfies~\eqref{primal-to-dual-variable}$, (x^\star, z^\star)$ is a pair of primal and dual optimal variables.

Similarly, under the assumption that there exists $\bar z \in E$ such that $f^\ast$ is continuous at $A^\ast\bar z$ and $g^\ast(-\bar z)<+\infty$, if $d$ is finite, and attained at $z^\star \in F$, then $\pi$ is attained at $x^\star\in E$ satisfying
\begin{equation}\label{dual-to-primal-variable}
    x^\star \in \partial f^\ast(A^\ast z^\star), \quad Ax^\star \in \partial g^\ast(-z^\star).
\end{equation}
Conversely,  if $(x^\star, z^\star)$ satisfies~\eqref{dual-to-primal-variable}$, (x^\star, z^\star)$ is a pair of primal and dual optimal variables.
\end{prpstn}

\subsection{Parametric convex optimisation}

\begin{lmm}
\label{parametric_optimisation}
Let $H$ be a Hilbert space, $Z$ be a metric space, $f : H \times Z \rightarrow \R\cup{\{+\infty\}}$.
Assume that 
\begin{itemize}
    \item $\forall \alpha \in Z$, $f(\cdot,\alpha)$ is convex on $H$,
     \item $\forall x \in H$, $f(x,\cdot)$ is continuous on $Z$,
    \item $f$ is sequentially weak-strong lower semicontinuous on $H \times Z$, \ie 
    \[\forall  x_n \rightharpoonup x, \; \forall \alpha_n \rightarrow \alpha, \quad f(x,\alpha) \leq \liminf_{n\rightarrow +\infty} f(x_n , \alpha_n),\]
    \item there exists a unique $x_\alpha \in H$ such that $\inf_{x \in H} f(x,\alpha) = f(x_\alpha, \alpha)$.
\end{itemize}
Then the mapping 
\[\alpha \in Z \longmapsto \inf_{x \in H} f(x,\alpha)\] is continuous on $Z$.
\end{lmm}
\begin{proof}
Let $\alpha_n \rightarrow \alpha$. Denoting 
$m(\alpha) = \inf_{x \in H} f(x,\alpha) = f(x_\alpha, \alpha)$, let us show that $m(\alpha_n)$ converges to $m(\alpha)$. 

\textit{Upper semicontinuity.}
For $x \in H$ fixed, thanks to the continuity of $f(x,\cdot)$, we pass to the limit in $f(x,\alpha_n) \geq m(\alpha_n)$ and find
\[m(\alpha) = \inf_{x \in H} f(x,\alpha) \geq \limsup_{n\rightarrow +\infty} m(\alpha_n).\]

\textit{Lower semicontinuity.}
We denote $x_n = x_{\alpha_n}$. Let  us for the moment admit that $(x_n)$ is bounded. Upon extraction, we may assume that $x_n \rightharpoonup \bar x$ for some $x \in H$. By sequential weak-strong lower semicontinuity,
\[f(\bar x, \alpha) \leq \liminf_{n \rightarrow +\infty} f(x_n, \alpha_n) = \liminf_{n \rightarrow +\infty} m(\alpha_n).\]
Since the left-hand side is bounded from below by $m(\alpha)$, we have proved lower semicontinuity (and in fact $\bar x = x_\alpha$).

We are left to proving the boundedness of $(x_n)$ to conclude the proof. Assume that $(x_n)$ is not bounded. Upon extraction, we may assume that
\[y_n := \frac{x_n}{\|x_n\|_H}\rightharpoonup y,\]
for some $y \in H$. For any fixed $\lambda>0$, we shall prove that $x_\alpha + \lambda y$ minimises $f(\cdot, \alpha)$, which contradicts the fourth assumption that there exists a single minimum point.

Indeed, we notice that
\[\Big(1-\frac{\lambda}{\|x_n\|_H}\Big) x_\alpha+  \frac{\lambda}{\|x_n\|_H} x_n\rightharpoonup x_\alpha+ \lambda y.\]
Hence, by weak-strong lower semicontinuity, convexity, the fact $x_n$ minimises $f(\cdot, \alpha_n)$ and continuity,
\begin{align*}f(x_\alpha+\lambda y, \alpha)& \leq \liminf_{n\rightarrow +\infty} f\Big(\Big(1-\frac{\lambda}{\|x_n\|_H}\Big) x_\alpha+  +\frac{\lambda}{\|x_n\|_H} x_n,\alpha_n\Big) \\
& \leq \liminf_{n\rightarrow +\infty}\Big(1-\frac{\lambda}{\|x_n\|_H}\Big)f(x_\alpha, \alpha_n) + \frac{\lambda}{\|x_n\|_H} f(x_n,\alpha_n)
\\& \leq \liminf_{n\rightarrow +\infty}\Big(1-\frac{\lambda}{\|x_n\|_H}\Big)f(x_\alpha, \alpha_n) + \frac{\lambda}{\|x_n\|_H} f(x_\alpha,\alpha_n) \\
& = \liminf_{n\rightarrow +\infty} f(x_\alpha,\alpha_n) = f(x_\alpha,\alpha).
\end{align*} 
\end{proof}

\section{The classical bathtub principle}
\label{app-sec-bathtub}

The classical bathtub principle characterises the maximisers, and gives the maximum value, of the constrained scalar product maximisation:
\begin{equation}
\label{bathtub-classical}\sup_{u \in \widetilde{\U}^\ast_L} \int_\Omega u(x) v(x) \,dx,\end{equation}
where $v\in L^2(\Omega)$ is arbitrary, and
\begin{equation*}
%\label{convex-constraint-set-classical}
\widetilde{\U}_L^\ast:= \left\{ u \in L^2(\Omega), \; 0 \leq u \leq 1  \text{ and } \int_\Omega  u= \mL\right\},\end{equation*}
is the convex hull (and $L^\infty$ weak-$\ast$ closure) of the set of characteristic functions whose support has the corresponding fixed measure:
\[\widetilde{\U}_L:=\{\chi_\omega, \quad \omega\subset \Omega, \quad |\omega|= \mL\}.\]

Recalling the notations \eqref{Phi} and \eqref{Phi-1} introduced in Section \ref{subsec-bathtub-convex-analysis}, the classical bathtub principle reads (we refer to \cite{Lieb2001}):

\begin{lmm}[classical bathtub principle]
\label{classical_bathtub}
Let $v \in L^2(\Omega)$. 
Denote $\rho(v):=\Phi^{-1}_v(\mL)$. The maximum in~\eqref{bathtub-classical} equals
\begin{equation*} 
\left(\int_{v>\rho(v)} v \right)+\rho(v)(\mL-|\{v>\rho(v)\}|) =\int_0^{\mL} \Phi_v^{-1},  \end{equation*}
 and the maximisers are given by
\begin{equation*}
u^\star := \chi_{\{v>\rho(v)\}}+c\chi_{\{v=\rho(v)\}},\end{equation*}
where  $c$ is any measurable function such that $0 \leq c \leq 1$ and 
\[
\int_{\{v=\rho(v)\}} c = \mL - |\{v>\rho(v)\}|.\]
In particular, if all the level sets of the function $v$ have zero measure, then the maximum is uniquely attained by
\begin{equation*} u^\star := \chi_{\{v>\rho(v)\}}, \end{equation*}
and the maximum hence equals $\displaystyle \int_{\{v>\rho(v)\}} v$.
\end{lmm}

Now, recall that we defined $\U_L$ by \eqref{1-shapess-L} and its convex hull $\overline{\U}_L$ by \eqref{convex-constraint-set} in Section \ref{subsec-bathtub-convex-analysis}. They are respective relaxations of
$\tilde{\U}_L$, and its convex hull
$\tilde{\U}^\ast_L$.

For $v \in L^2(\Omega)$, we consider the relaxed version of \eqref{bathtub-classical}:
\begin{equation}
\label{bathtub-relaxed-app}\sup_{u \in \overline{\U}_L} \int_\Omega u(x) v(x) \,dx.\end{equation}

Then, the complete solution of Lemma \ref{relaxed_bathtub} is given by the following:
    \begin{lmm}[relaxed bathtub principle]
    \label{lmm-app-bathtub-relaxed}
Let $v\in L^2(\Omega)$. Denote $h(v)=  \max(0, \Phi_v^{-1}(\mL)) = \max(0,\rho(v))$. Then, the maximum in~\eqref{bathtub-relaxed-app} equals
 \[\int_0^{\min(\Phi_v(0), \mL)} \Phi_v^{-1},\]
and the maximisers are given by
\begin{equation*}u^\star := \chi_{\{v>h(v)\}}+c\chi_{\{v=h(v)\}},\end{equation*}
where  $c$ is any measurable function such that $0 \leq c \leq 1$ and 
\[\left\{\begin{aligned}
\int_{\{v=h(v)\}} c = \mL - |\{v>h(v)\}| \ &\textrm{ if } \ h(v) >0 \\
\int_{\{v=h(v)\}} c \leq \mL - |\{v>h(v)\}| \ &\textrm{ if } \ h(v) =0
\end{aligned}\right.\]
\end{lmm}
\begin{rmrk}\label{app-rmrk-bathtub-uniqueness}
In particular, if $h(v)>0$, there is a unique maximiser in~\eqref{bathtub-relaxed-app} if and only if
\[|\{v>h(v)\}|+|\{v=h(v)\}|=\mL.\]
Indeed, when the above does not hold, $c$ can be chosen to have values in $(0,1)$ on a set of nonzero measure, and the maximisers are no longer unique. 

On the other hand, if $h(v)=0$, as soon as $|\{v=h(v)\}|>0$ the maximisers are no longer unique.
\end{rmrk}

\begin{proof}
\begin{comment}
    sketch of proof:
    \begin{enumerate}
        \item If $u^\star$ is a maximiser, then
        \[\int_{\Omega} u \geq \min(|\{v>0\}|, L|\Omega|).\]
        Key: $u=0$ on $\{v<0\}$.
        \item Case 1: $|\{v>0\}| \geq L|\Omega|$. 
        
        Any maximiser satisfies
        \[\int_\Omega u=L|\Omega|.\]
        Back to the classical bathtub principle.
        \item Case 2: $|\{v>0\} |\leq L|\Omega|$.
Then $\chi_{\{v>0\}}$ is an admissible test function and we have
\[\int_\Omega u^\star v =\int_{v>0} u^\star v \geq \int_{v>0} v.\]
Now, given the $L^\infty$ bounds on $u^\star$, 
\[\int_{v>0} v \leq \int_{\Omega} u^\star v \leq \int_{v>0} v,\]
and it follows that $u^\star=1$ on $\{v>0\}.$     
    \end{enumerate}
\end{comment}
  Let us first note that the supremum exists and is attained, as \eqref{bathtub-relaxed-app} consists in maximising a continuous function on a weak-$\ast$ compact set. Also note that any maximiser $u^\star$ obviously satisfies $\operatorname{supp} u^\star \subset \{v\geq 0\}$.

    To prove the relaxed bathtub principle, we distinguish two cases:
    \begin{enumerate}

        \item  We first consider the case where $|\{v> 0\}|\geq L|\Omega|$.
        
        Let 
        $u^\star$ be a maximiser. 
        
        Suppose 
        \begin{equation}\label{unsaturated-integral}\int_\Omega u^\star < L|\Omega|.\end{equation}
        
        If $|\{u^\star<1\} \cap \{v>0\}|>0$, then there exists a set $\omega\subset \{u^\star<1\}\cap \{v>0\}$ of nonzero measure with $|\omega|\leq L|\Omega|$, so that
        \[0<\int_\omega 1-u^\star \leq L|\Omega|-\int_\Omega u^\star.\]
        Then, $\int_\Omega u^\star+\chi_\omega(1-u^\star) \leq L|\Omega|$ and
        \[\int_\Omega (u^\star+\chi_\omega (1-u^\star)) v > \int_\Omega u^\star v,\]
      which contradicts the fact that $u^\star$ is a maximiser. Thus $u^\star=1$ a.e. on $\operatorname{supp}u^\star \cap \{v>0\}$. 

        Then, the assumption \eqref{unsaturated-integral} implies $|\operatorname{supp}u^\star \cap \{v>0\}| <L|\Omega|$, so that there exists a measurable set $\omega$ satisfying $|\omega|\leq L|\Omega|$ and \[\operatorname{supp} u^\star \cap\{v>0\}\subsetneq \omega \subset \{v>0\}.\]
        We then have
        \[\int_\Omega u^\star v = \int_{\operatorname{supp} u^\star \cap \{v>0\}} v < \int_\omega v =\int_\Omega \chi_\omega v\]
        which contradicts the fact that $u^\star$ is a maximiser.
        
        By contradiction we have thus proved that any maximiser $u^\star$ satisfies $\textstyle \int_\Omega u^\star=L|\Omega|$, so that the relaxed problem reduces to the classical bathtub problem. 

        The assumption on $v$ implies that $\rho(v)\geq0$. Hence, if $h(v)=0$, then $\rho(v)=0$ \ie $|\{v>0\}|=L|\Omega|$ and $|\{v=0\}|=0$. Thus applying the classical bathtub principle yields the unique maximiser $u^\star=\chi_{\{v>0\}}$.
        On the other hand, if $h(v)>0$, $\rho(v)=h(v)$ and  a straightforward application of the classical bathtub principle yields the maximisers $\chi_{\{v>h(v)\}}+c\chi_{\{v=h(v)\}}$,
        where $c$ is a measurable function such that $0\leq c \leq 1$ and
        \[\int_{\{v=h(v)\}} c=L|\Omega|-|\{v>h(v)\}|.\]

          \item We now turn to the case where $|\{v> 0\}|< L|\Omega|$. This implies that $h(v)=0$.

          %Again, it is clear that any maximiser of the relaxed problem \eqref{bathtub-relaxed-app} is supported in $\{v\geq 0\}.$

          Let $u^\star$ be a maximiser. From the assumption on $v$ and the constraints on $u^\star$, $\textstyle \int_{\{v>0\}} u^\star <L|\Omega|$. From the same argument as above, $u^\star=1$ on $\{v>0\}$.

          Finally, the values of $u^\star$ on $\{v=0\}$ do not affect the quantity in \eqref{bathtub-relaxed-app}. The only requirement on $u^\star_{|\{v=0\}}$ is that it be measurable, and
          \[\int_{\{v=0\}} u^\star \leq L|\Omega|-\int_{\{v>0\}} u^\star=L|\Omega| - \{v>0\},\]
          in order to satisfy the integral constraint on $u^\star.$
         We can thus write $u^\star$ as $\chi_{\{v>0\}}+c \chi_{\{v=0\}}$
         with $0\leq c \leq 1$ and 
         \[\int_{\{v=0\}} c \leq L|\Omega|-|\{v>0\}|,\]
         which concludes the proof.
    \end{enumerate}
\end{proof}

\end{document}